\documentclass[10pt,journal,compsoc]{IEEEtran}

\usepackage{commath}
\usepackage{amsmath,amsfonts,amssymb}
\usepackage{array,arydshln}
\usepackage{graphicx,subfigure}
\usepackage{amsthm}
\usepackage{algorithm}
\usepackage{algorithmic}
\usepackage{color}
\usepackage{multirow}
\usepackage{mathtools}
\usepackage{makecell}
\usepackage{caption}
\usepackage{fancyvrb}
\usepackage{url}
\ifCLASSINFOpdf
  \usepackage[pdftex]{thumbpdf}
\else
  \usepackage[dvips]{thumbpdf}
\fi

\ifCLASSOPTIONcompsoc
  \usepackage[nocompress]{cite}
\else
  \usepackage{cite}
\fi
\newtheorem{remark}{Remark}
\newtheorem*{remark*}{Remark}
\newtheorem{theorem}{Theorem}
\newtheorem{lemma}{Lemma}

\allowdisplaybreaks

\begin{document}

\title{ {A Fast Frequent Directions Algorithm \\ for Low Rank Approximation}}

\author{Dan~Teng, and Delin~Chu~\IEEEmembership{(Senior Member,~IEEE)}
\IEEEcompsocitemizethanks{\IEEEcompsocthanksitem D. Teng and D. Chu are with the Department of Mathematics, National University of Singapore, Block S17, 10 Lower Kent Ridge Road, Singapore, 119076. \protect E-mail: tengdan@u.nus.edu, matchudl@nus.edu.sg. \ This work was supported in part by NUS research grant R-146-000-236-114.}
\IEEEcompsocitemizethanks{\IEEEcompsocthanksitem DOI: 10.1109/TPAMI.2018.2839198}}

\IEEEtitleabstractindextext{%
\begin{abstract}
Recently a deterministic method, frequent directions (FD) is proposed to solve the high dimensional low rank approximation problem.
It works well in practice, but experiences high computational cost.
In this paper, we establish a fast frequent directions algorithm for the low rank approximation problem,  which implants
a randomized algorithm, sparse subspace embedding (SpEmb) in FD. This new algorithm makes
use of FD's natural block structure and sends more information through SpEmb to each block in FD. We prove that our new algorithm produces a good low
rank approximation with a sketch of size linear on the rank approximated. Its effectiveness and efficiency are demonstrated by
the experimental results on both synthetic and real world datasets, as well as applications in network analysis.
\end{abstract}

}

\maketitle

\IEEEdisplaynontitleabstractindextext

\IEEEpeerreviewmaketitle

\ifCLASSOPTIONcompsoc
\IEEEraisesectionheading{\section{Introduction}\label{sec: intro}}
\else
\section{Introduction}
\label{sec: intro}
\fi

\IEEEPARstart{G}{iven} an input matrix $A\in\mathbb{R}^{n\times d}$ ($n\geq d$) and a rank parameter $k\leq$ rank($A$), the low rank approximation problem is to find a matrix $\tilde{A}_k\in\mathbb{R}^{n\times d}$  of rank-$k$ to approximate the original matrix $A$.
Low-rank approximation is a major topic in scientific computing and it is an essential tool in many applications including principal component analysis, data compression, spectral clustering and etc \cite{Rokhlin2009, Halko2011, Dhanjal2014}.

The best low rank approximation can be obtained through singular value decomposition (SVD) \cite{Golub2013}, rank revealing QR factorization \cite{Chan1994}
or two-sided orthogonal factorizations \cite{Fierro1997}. However, with rapid growth on data dimension, these traditional and deterministic methods are expensive and inefficient \cite{Mahoney2011}. For a matrix $A\in\mathbb{R}^{n\times d}$ $(n\geq d)$, the SVD computation requires $O(nd^2)$ time which is unfit for large $n$ and $d$. Many randomized algorithms have been developed to target at this large dimension issue. The idea of a randomized algorithm is to trade accuracy with efficiency. A good randomized algorithm tends to have high accuracy with low running time.
Random Sampling and random projection are two commonly used randomization techniques which can be easily employed to solve the low rank approximation problem. Random sampling finds a small subset of rows or columns based on a pre-defined probability distribution to form a sketch and it is also known as the column subset selection problem \cite{Boutsidis2011}.
Possible techniques include subspace sampling \cite{Drineas20063}, norm sampling \cite{Drineas2003, Holodnak2015, Drineas2006} and adaptive sampling \cite{Wang2013, Deshpande2006}.
Random projection combines rows or columns together to produce a small sketch \cite{Liberty2007, Sarlos2006}.
Possible techniques include subsampled randomized Hadamard transform \cite{Tropp2011, Woolfe2008, Boutsidis2013}, discrete Fourier transforms \cite{Avron2011, Urano2013}, sparse subspace embedding (SpEmb) \cite{Clarkson2013, Meng2013, Nelson2013}, Gaussian projection \cite{Dasgupta2003, Gu2014}, subspace iterations \cite{Halko2011, Gu2014, Musco2015} and the fast Johnson-Lindenstrass transform \cite{Ailon2009}. Among them, SpEmb is easy to implement, respects the sparsity of the input matrix and achieves similar error bound with the lowest running time.

Apart from randomized algorithms, another approach to produce a sketch of a matrix is through frequent directions (FD) \cite{Liberty2013,
Ghashami2015}.
It is motivated by the idea of frequent items in item frequency approximation problem. This algorithm requires the sketch size to be linear on the rank approximated to reach the relative-error bound for the low rank approximation. It works well in practice, but with high computational cost.
An additional point about FD is that besides producing the sketch, it also returns a matrix consisting of its right singular vectors during the process. This is useful for low rank approximation which we will discuss in later sections.

\emph{Is there any method that achieves high accuracy comparable to FD and low computational cost similar as SpEmb for the low rank approximation problem?} In this paper, we establish a fast frequent directions algorithm, which provides an answer to this question.  This new algorithm
performs FD by sending more information to each iteration which makes the SVD computation in each iteration more valuable and accurate. In the meanwhile,
it can be seen as performing SpEmb on the input matrix to obtain an intermediate sketch, then applying FD on the sketch which reduces the number of iterations required. Our new algorithm preserves the high accuracy of FD with the sketch size required to reach a good error bound
linear on the rank approximated; it also reduces the computational cost substantially by cutting down the number of iterations in FD and exploiting the
sparsity of the input matrix.

The rest of this paper is organized as follows: We review some related work on low rank approximation, SpEmb and FD in Section~\ref{sec: 2}, they provide the foundation and motivation in developing the new algorithm. Then we introduce our new algorithm with theoretical analysis in Section~\ref{sec: 3}.
In Section~\ref{sec: 4}, we evaluate the performance of our new algorithm in comparison with FD and three well-known randomized algorithms including SpEmb
on both synthetic and real world datasets. We explore the application on network problems in Section~\ref{sec: 5}. Finally, concluding remarks are given in Section~\ref{sec: 6}.

\section{Related Work}
\label{sec: 2}

Given an input matrix $A\in\mathbb{R}^{n\times d}$, let
$A = U_A\Sigma_A V_A^T$
be the singular value decomposition of $A$, where the singular values of $A$ are arranged in a non-increasing order in $\Sigma_A$. By the Eckart-Young-Mirsky theorem, the best rank-$k$ ($k\leq$ rank($A$)) approximation for both Frobenius norm ($F$-norm) and spectral norm ($2$-norm) is
$A_k = (U_A)_k(\Sigma_A)_k(V_A)_k^T$,
where $(\Sigma_A)_k\in\mathbb{R}^{k\times k}$ consists of the top $k$ singular values in its diagonal entries, $(U_A)_k\in\mathbb{R}^{n\times k}$ and $(V_A)_k\in\mathbb{R}^{d\times k}$ consist of the $k$ left and $k$ right singular vectors of $A$ which correspond to the top $k$ singular values, respectively.
For randomized algorithms and FD method, they aim to find a low rank approximation  $\tilde{A}_k$ to minimize $\norm{A-\tilde{A}_k}^2_F$.

To construct $\tilde{A}_k$, the first step is to produce a sketch of the input matrix and it could be
formulated as premultiplying a sketching matrix $S\in\mathbb{R}^{\ell\times n}$ to $A$ with sketch size $\ell\ll n$. In particular, for Gaussian projection, $S_{ij}\sim\mathcal{N}(0,1)$ i.i.d. with $i\in[\ell], j\in[n]$;
 for uniform sampling, $S$ consists of $\ell$ rows uniformly selected from $I_n$. $S$ is designed to form $SA$ efficiently and also to preserve the important information in $A$. It can be found in \cite{Mahoney2011, Woodruff2014, Drineas2006} on how such an $S$ is constructed. For the deterministic method FD, the construction of a sketch is discussed in Subsection~\ref{sec: FD}.

After forming the sketch $B = SA\in\mathbb{R}^{\ell\times d}$, the next step is to construct the low rank approximation $\tilde{A}_k$. Let $\Pi_{B, k}^F(A)\in\mathbb{R}^{n\times d}$ denote the best rank-$k$ approximation to $A$ in the row space of $B$, with respect to $F$-norm.
It has been proved \cite{Boutsidis2011, Boutsidis2013, Woodruff2014} that
\begin{align}
\label{eq: Fproof}
\norm{A-[AV]_kV^T}_F^2 &= \norm{A-\Pi_{B, k}^F(A)}_F^2
\end{align}
where $V\in\mathbb{R}^{d\times \ell}$ satisfies $V^TV = I_\ell$ and its columns span the row space of $B$ and $[AV]_k$ denotes the best rank-$k$ approximation of $AV$. Thus, $\tilde{A}_k$ could be set as $[AV]_kV^T$.

\subsection{Sparse Subspace Embedding (SpEmb)}
\label{sec: SpEmb}
Since its introduction by Clarkson and Woodruff \cite{Clarkson2013}, SpEmb has received great attention due to its efficient time complexity. It has been extensively used as a randomization technique to solve numerical linear algebraic problems,
here we look at its application on low rank approximation.

Define a sketch matrix $S = \phi D\in \mathbb{R}^{\ell\times n}$, such that
\begin{itemize}
\item $\phi\in \{0,1\}^{\ell\times n}$ is an $\ell\times n$ binary matrix with $\phi_{h(i), i} = 1$ and remaining entries $0$, where $i\in[n]$ and $h: [n]\rightarrow [\ell]$ is a random map so that for each $i\in[n]$, $h(i) = \gamma$ for $\gamma \in[\ell]$ with probability $\frac{1}{\ell}$;
\item $D$ is an $n\times n$ diagonal matrix with each diagonal entry independently chosen to be $\pm 1$ with equal probability.
\end{itemize}

$S$ is called a SpEmb matrix, each column of $S$ contains exactly one nonzero entry $-1$ or $+1$, with equal probability. Due to the structure of $S$,
the sketch $SA$ can be formed efficiently by exploiting the sparsity of $A$  \cite{Clarkson2013, Woodruff2014, Nelson2013, Meng2013}. SpEmb is described in Algorithm~\ref{alg: proto}.
\begin{algorithm}
\caption{SpEmb}
\label{alg: proto}
\begin{algorithmic}[1]
\REQUIRE $A\in\mathbb{R}^{n\times d}$, rank parameter $k$, sketch size $\ell$.
\ENSURE $\tilde{A}_k$.
\STATE Form the sketch matrix $B= SA\in\mathbb{R}^{\ell\times d}$ with $S\in\mathbb{R}^{\ell\times n}$ being a SpEmb matrix;
\STATE Construct a $d\times \ell$ matrix $V$ whose columns form an orthogonal basis for the row space of $B$;
\STATE Compute $[AV]_k$ and  $\tilde{A}_k = [AV]_kV^T$.
\end{algorithmic}
\end{algorithm}
We can use the QR factorization $B^T = VR$ to construct $V$, while $[AV]_k$ could be obtained by the SVD on $AV$. The reader may refer to  \cite{Clarkson2013, Meng2013, Woodruff2014} for more details on SpEmb.

\subsection{Frequent Directions (FD)}
\label{sec: FD}
Compared to most randomized algorithms, FD approaches the matrix sketching problem from a different perspective. It is deterministic, space efficient and produces more accurate estimates given a fixed sketch size.
 It was introduced by Liberty \cite{Liberty2013}, later extended by Ghashami et al. \cite{Ghashami2015} with a comprehensive analysis, and a few variations \cite{Desai2015, Ghashami2016} have been developed since then.
 The algorithm extends the idea of frequent items for item frequency approximation problem to a general matrix. Given an input matrix $A\in\mathbb{R}^{n\times d}$ and space parameter $\ell$, at first, the algorithm considers the first $2\ell$ rows in $A$ and shrinks its top $\ell$ orthogonal vectors by the same amount to obtain an $\ell\times d$ matrix;
 then combines them with the next $\ell$ rows in $A$ for the next iteration, and repeat the procedure. It is illustrated in Fig.~\ref{fig: fd}. The detailed algorithm is given in Algorithm~\ref{alg: fd}.

\begin{figure*}[tb]
\centering
\includegraphics[width = 12cm, height = 4.15cm]{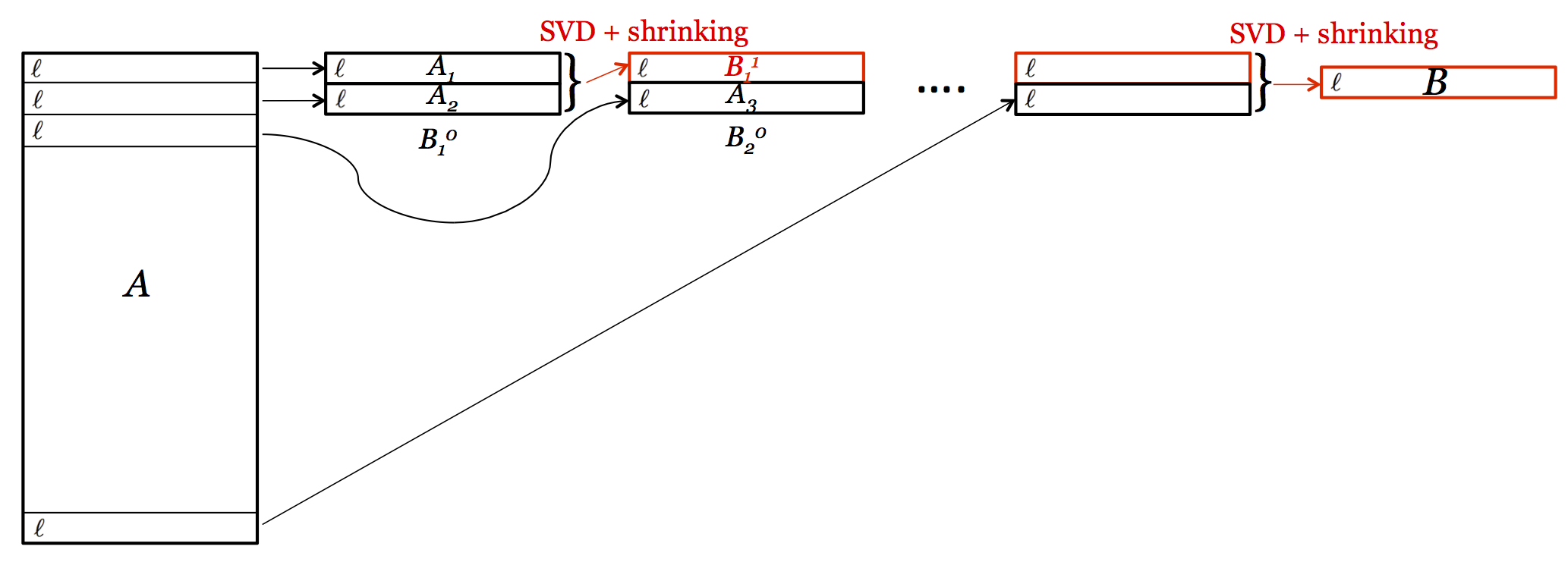}
\caption{Illustration of FD.}
\label{fig: fd}
\end{figure*}

\begin{algorithm}
\caption{FD}
\label{alg: fd}
\begin{algorithmic}[1]
\REQUIRE $A\in\mathbb{R}^{n\times d}$, rank parameter $k$, sketch size $\ell$.
\ENSURE $\tilde{A}_k$ and $B$.
\STATE Set $B\leftarrow 0^{2\ell\times d}$, $B(1:\ell,:) \leftarrow A(1:\ell, :)$;
\STATE Append zero rows to $A$ to ensure the number of rows in $A$ is a multiple of $\ell$;
\STATE \textbf{for} $i = 1, \hdots, \lceil n/\ell\rceil-1$, \textbf{ do}\\
		~~~~ $B(\ell+1:2\ell, :) \leftarrow A(i\ell + 1: (i+1)\ell, :)$\\
		~~~~ $[U, \Sigma, V] = SVD(B)$, \\
		~~~~ $\delta_i = \sigma_{\ell+1}^2$, \textit{~~~~~~~~~~~~~~~~~~~~~~~~~~~~~~~~~~~~~} \# ~ $\sigma_{\ell+1} = \Sigma_{\ell+1, \ell+1}$\\
		~~~~ $B \leftarrow \sqrt{\max(\Sigma^2-I_{2\ell}\delta_i, 0)}\cdot V^T$;\\ ~~~~~\textit{~~~~~~~~~~~~~~~~~~~~~~~~~~~\# ~The last $\ell$ rows of $B$ are zero-valued} \\
\STATE Set $B=B(1:\ell, :)$ and $V = V(:, 1:\ell)$;
\STATE Compute $[AV]_k$ and $\tilde{A}_k = [AV]_kV^T$.
\end{algorithmic}
\end{algorithm}

Algorithm~\ref{alg: fd} is slightly different from the original one in
\cite{Liberty2013, Ghashami2015} in the sense that it
performs one more round of SVD and $B$ update at the end when $B$ contains zero rows. This is to maintain a sketch $B$ of size $\ell\times d$ and to acquire a matrix $V$ whose columns form an orthonormal basis for the row space of $B$.

\section{Fast Frequent Directions via Sparse Subspace Embedding (SpFD)}
\label{sec: 3}
Notice that each iteration of FD requires $O(d\ell^2)$ running time to compute the SVD and there are $\lceil n/\ell\rceil-1$ iterations. The main cost comes from the large number of iterations when $n$ is extremely large. Here note at the $i$th iteration ($i>1$), each $B_i^0$
is composed of the approximate $B_{i-1}^1$ from the previous step and a new set of data in $A$, from which FD extracts and shrinks the top $\ell$ important right singular vectors to form $B_i^1$ (refer to Fig.~\ref{fig: fd}).

The goal is to ensure that the top $k$ right singular vectors of $A$ are preserved through every iteration. If the number of iterations is large, there is a higher risk of discarding useful information during the extractions via SVD with shrinking. Since each iteration is considered as an approximation step, to make this approximation more efficient, instead of bringing just $\ell$ rows from $A$ to form $B_i^0$ in the $i$th iteration, we bring more information of $A$ to $B_i^0$ in a compact way (squeeze more rows of $A$ in the lower $\ell$ rows of $B_i^0$). This could make each SVD computation more valuable and reduce the number of iterations which in a way lessens the approximation error. It obviously introduces a new type of approximation error as the information being analyzed in each iteration is no longer of $A$, but rather a sketch of $A$. Therefore, it requires balancing of the sketch error from compressing more information in $\ell$ rows during each iteration and the shrinking error from performing more iterations of SVD computations, i.e., balancing the error incurred in the sketching and the error incurred in FD.

Here SpEmb is used to perform the intermediate sketching. It is efficient and respects the possible sparsity of the input matrix.

\begin{figure*}[tb]
\centering
\includegraphics[width = 16cm, height = 4.3cm]{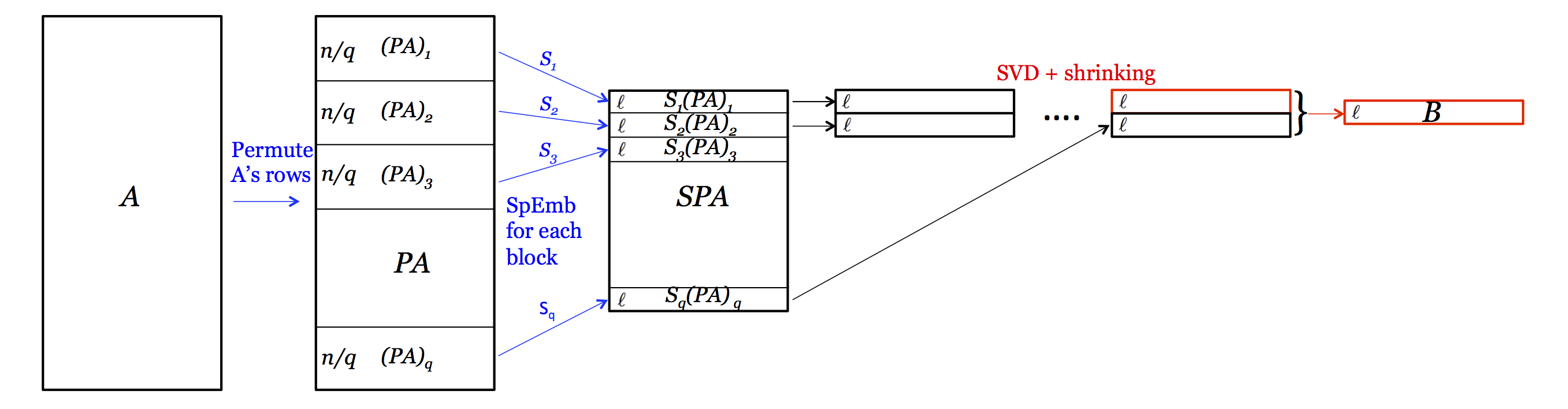}
\caption{Illustration of SpFD.}
\label{fig: intuition}
\end{figure*}

Now let's illustrate the method in detail. Define
\begin{equation}
S = \begin{bmatrix}
S_1 & & \\ & \ddots & \\ & & S_q
\end{bmatrix}\in\mathbb{R}^{q\ell\times n},
\label{eq: S}
\end{equation}
where each $S_i$ ($i=1,\hdots, q$) is a SpEmb matrix of dimension $\ell \times n/q$. Here we assume $n$ is a multiple of $q$. If not, we can append zero rows to the end of $A$ until $n$ is. Let $P\in\mathbb{R}^{n\times n}$ denote a random permutation matrix constructed via Fisher-Yates-Knuth shuffle, for each $e_j$ (the $j$th column of $I_n$) with $j\in[n]$, $\mathbb{P}[Pe_j = e_i] = 1/n$ for $i \in[n]$. $P$ is generated using \texttt{randperm} in MATLAB.
Then we apply FD to the intermediate sketch $SP A\in\mathbb{R}^{q\ell\times d}$ to obtain $B\in\mathbb{R}^{\ell\times d}$ and the corresponding $V\in\mathbb{R}^{d\times \ell}$. Fig.~\ref{fig: intuition} illustrates the process.

SpFD is a combination of FD and SpEmb.
Each $S_i$ maps $n/q$ rows of $PA$ to $\ell$ rows in the intermediate sketch through SpEmb,
then perform FD on the intermediate sketch $SPA$. The magnitude of $q$ indicates the weight distributed between FD and SpEmb. When $q = 1$, $S$ has only one block, then SpFD becomes SpEmb. When $q = n/\ell$, $S$ is of dimension $n\times n$, SpFD will become FD if $SPA$ contains the same rows as $A\in\mathbb{R}^{n\times d}$. Here note as a SpEmb matrix, each $S_i\in\mathbb{R}^{\ell\times \ell}$ may contain many zero rows and so will $S\in\mathbb{R}^{n\times n}$ and the sketch $SPA\in\mathbb{R}^{n\times d}$ which implies SpFD is unlikely to become FD. This will be discussed more while analyzing SpFD's running time in Section~\ref{sec: 4}.
The new method SpFD is described in Algorithm~\ref{alg: SpFD}.
\begin{algorithm}[!]
\caption{SpFD}
\label{alg: SpFD}
\begin{algorithmic}[1]
\REQUIRE $A\in\mathbb{R}^{n\times d}$, sketch size $\ell$, rank parameter $k$,  number of blocks $q$ (assume $n$ is a multiple of $q$).
\ENSURE $\tilde{A}_k$ and $B$.
\STATE Apply random permutation $P$ on $A$ to obtain $PA$;
\STATE Set $B \leftarrow 0^{2\ell\times d}$, $B(1:\ell,:)\leftarrow SpEmb(PA(1:\frac{n}{q},: ), \ell)$;\\
 \textit{~~~~~~~~~~~~~~~~~~~~~~~~~~~~~~\# ~This is to set $B_1^0(1:\ell,:) =S_1(PA)_1$}
\STATE \textbf{for} $i = 1, ..., q-1$ \textbf{do} \\
		~~~~ $B(\ell+1:2\ell, :) \leftarrow SpEmb(PA(i \frac{n}{q}+1: (i+1)\frac{n}{q}, :), \ell)$, \\
\textit{~~~~~~~~~~~~~~~~~~~~~\#~ This is to set $B_i^0(\ell+1:2\ell, :) = S_{i+1}(PA)_{i+1}$}\\
		~~~~ $[U, \Sigma, V] \leftarrow SVD(B)$,\\
		~~~~ $\delta_i = \sigma^2_{\ell+1}$,\\
		~~~~ $B = \sqrt{\max(\Sigma^2-I_{2\ell}\delta_i, 0 )}\cdot V^T$,\\ 
\STATE Set $B=B(1:\ell, :)$ and $V = V(:, 1:\ell)$ (if $q=1$, construct $V$ directly from $B$);
\STATE Compute $[AV]_k$ and  $\tilde{A}_k = [AV]_kV^T$.
\end{algorithmic}
\end{algorithm}

The random permutation matrix $P$ is constructed to mix the blocks of $A$. Denote $A = [a_1 ~  a_2 ~ \hdots ~ a_n]^T$ where $a_i\in\mathbb{R}^d$ and let $[i_1 ~ i_2 ~ \hdots ~ i_n] = $\texttt{ randperm} ($n$), then $PA = [a_{i_1} ~ a_{i_2} ~ \hdots ~ a_{i_n}]^T$, it is only required to load $(PA)_j = [a_{i_{(j-1)n/q+1}} ~ a_{i_{(j-1)n/q+2}} ~ \hdots ~ a_{i_{jn/q}}]^T$ sequentially for each $j = 1, \hdots, q$.  Therefore, although SpFD requires a permutation step, it is not necessary to load the whole data in memory.

\subsection{Main Theorem}
\label{sec: LRmaintheorem}
Next we are going to prove SpFD's accuracy. We first look at  $SP\in\mathbb{R}^{q\ell\times n}$; show that if $q\ell$ satisfies certain lower bound, $SP$ is a good sketching matrix.
$SP$ is in fact a variation of a SpEmb matrix. The difference is that $SP$ does not allow more than $n/q$ rows mapping to the same row in the sketch.

\begin{lemma}
\label{lemma: SpFD1}
Given a matrix $U\in\mathbb{R}^{n\times k}$ with orthonormal columns $(k\leq n)$. For any $\varepsilon, \delta\in(0,1)$, let $S\in\mathbb{R}^{q\ell\times n}$ be defined as (\ref{eq: S})
with
\[
q\ell \geq \dfrac{k^2+k}{\varepsilon^2\delta}.
\]
Let $P$ be an $n\times n$ random permutation matrix (i.e., for each $e_j$ with $j\in [n]$, $\mathbb{P}[Pe_j = e_i] = 1/n$ for $i\in[n]$), then with probability at least $1-\delta$,
\begin{equation}
\norm{U^T(SP)^TSPU-I}_2 \leq \varepsilon.
\label{eq: bound}
\end{equation}
\end{lemma}

\begin{proof}
To show (\ref{eq: bound}), we bound $\mathbb{E}\left[\norm{U^T(SP)^TSPU-I}_F^2\right]$ and then apply Markov's inequality.

Let $X = U^T(SP)^TSPU\in\mathbb{R}^{k\times k}$ and $x_{st}$ denote the $(s,t)$th-entry of $X$. Let $\mathcal{L}_i$ denote the set of integers from domain $\left[(\lceil\frac{i}{\ell}\rceil-1)\frac{n}{q}+1, \lceil\frac{i}{\ell}\rceil \frac{n}{q} \right]$, the cardinality of $\mathcal{L}_i$ is $\frac{n}{q}$. Note that for $i \in \{C\ell+1, \hdots, (C+1)\ell\}$ with an integer $C\geq 0$, the corresponding $\mathcal{L}_i$'s are the same.
Here note $S\in\mathbb{R}^{q\ell\times n}$ can be written as 
\[
\small
S = \begin{bmatrix}S_1 & & \\ & \ddots & \\ & & S_{q}\end{bmatrix} = \begin{bmatrix}\phi_1 & & \\ & \ddots & \\ & & \phi_{q}\end{bmatrix}\begin{bmatrix}D_1 & & \\ & \ddots & \\ & & D_{q}\end{bmatrix} =: \phi D
\]
where each $S_\alpha = \phi_\alpha D_\alpha\in\mathbb{R}^{\ell\times n/q}$ ($\alpha\in[q]$) is a SpEmb matrix defined as in Section~\ref{sec: SpEmb}.
Note that $\phi$ is a block diagonal matrix and $\phi_{ij} = 0$ when $j\notin \mathcal{L}_i$. For $1\leq s,t \leq k$,
\begin{align*}
x_{st} =~& \sum_{i=1}^{q\ell}\left( \sum_{j=1}^n\phi_{ij}d_j(PU)_{js} \right)\left(\sum_{j=1}^n \phi_{ij}d_j(PU)_{jt}\right)\\
=~& \sum_{i=1}^{q\ell}\sum_{j_1,j_2 =1}^n\phi_{ij_1}\phi_{ij_2} d_{j_1}d_{j_2}(PU)_{j_1s}(PU)_{j_2t},
\end{align*}
and its expectation:
\begin{align*}
\mathbb{E}[x_{st}]
 =& \sum_{i=1}^{q\ell}\sum_{j_1,j_2 = 1}^n \mathbb{E}[\phi_{ij_1}\phi_{ij_2}]\mathbb{E}[d_{j_1}d_{j_2}]\mathbb{E}[(PU)_{j_1s}(PU)_{j_2t}] \\
=& \sum_{i=1}^{q\ell}\sum_{j=1}^n\mathbb{E}[\phi_{ij}^2]\mathbb{E}[(PU)_{js}(PU)_{jt}] \\
=& \sum_{i=1}^{q\ell}\sum_{j\in \mathcal{L}_i} \mathbb{E}[\phi_{ij}^2]\cdot \dfrac{1}{n}\sum_{j=1}^nU_{js}U_{jt}\\
=& q\ell\cdot \dfrac{n}{q}\dfrac{1}{\ell}\cdot \dfrac{1}{n}\sum_{j=1}^nU_{js}U_{jt}\\
=& \delta_{st}.
\end{align*}
The second equality holds since $\mathbb{E}[d_{j_1}d_{j_2}] = 1$ if $j_1=j_2$ and $0$ otherwise; the third equality is because $\phi_{ij} = 0$ for $j\notin \mathcal{L}_i$ and $\mathbb{E}[(PU)_{js}(PU)_{jt}] = 1/n \sum_{j=1}^n U_{js}U_{jt}$.
We also need
\begin{align*}
 \mathbb{E}[x_{st}^2]
=& \sum_{i,i' =1}^{q\ell}\sum_{j_1, j_2, j_1', j_2'=1}^n\mathbb{E}[\phi_{ij_1}\phi_{ij_2}\phi_{i'j_1'}\phi_{i'j_2'}] ~\cdot \\
& \mathbb{E}[d_{j_1}d_{j_2}d_{j_1'}d_{j_2'}]
	\mathbb{E}[(PU)_{j_1s}(PU)_{j_2t}(PU)_{j_1's}(PU)_{j_2't}].
\end{align*}
Note that
\begin{eqnarray*}
\mathbb{E}\left[d_{j_1}d_{j_2}d_{j_1'}d_{j_2'}\right] &=& \left\{\begin{alignedat}{2}
			&1& ~ ~~~ &j_1 = j_2 = j_1' = j_2',\\
			&1& ~ ~~~ &j_1 = j_2 \neq j_1' = j_2',\\
			&1& ~ ~~~ &j_1 = j_1' \neq j_2 = j_2',\\
			&1& ~ ~~~ &j_1 = j_2' \neq j_1' = j_2,\\
			&0& ~ ~~~ &\text{otherwise,}
			\end{alignedat}\right.
\end{eqnarray*}
thus it is required to discuss $\mathbb{E}[\phi_{ij_1}\phi_{ij_2}\phi_{i'j_1'}\phi_{i'j_2'}]$ and $\mathbb{E}[(PU)_{j_1s}(PU)_{j_2t}(PU)_{j_1's}(PU)_{j_2't}]$ for each case:
\\
\\
\textbf{1.} $j_1 = j_2 = j_1' = j_2'$: Note that each column of $\phi$ has exactly one $1$, so $\phi_{ij}^2\phi_{i'j}^2 = 0$ if $i\neq i'$, then
\begin{align}
\label{eq: app1}
~& \sum_{i,i'=1}^{q\ell}\sum_{j=1}^n \mathbb{E}[\phi_{ij}^2\phi_{i'j}^2] \mathbb{E}[(PU)_{js}^2(PU)_{jt}^2]\nonumber\\
=~& \sum_{i=1}^{q\ell}\sum_{j=1}^n \mathbb{E}[\phi_{ij}^4]\cdot \dfrac{1}{n}\sum_{j=1}^nU_{js}^2U_{jt}^2 ~=~ \sum_{j=1}^nU_{js}^2U_{jt}^2.
\end{align}
\textbf{2.} $j_1 = j_2 \neq j_1' = j_2'$:
\begin{align*}
~&\mathbb{E}[(PU)_{js}(PU)_{jt}(PU)_{j's}(PU)_{j't}] \\
=~& \dfrac{1}{n(n-1)}\sum_{j\neq j'}U_{js}U_{jt}U_{j's}U_{j't}~;\\
~&\sum_{i,i'=1}^{q\ell}\sum_{j\neq j'}^n\mathbb{E}[\phi_{ij}^2\phi_{i'j'}^2] \\
=~& \sum_{i=1}^{q\ell}\sum_{j\neq j'}^n\mathbb{E}[\phi_{ij}^2\phi_{ij'}^2]+\sum_{i\neq i'}^{q\ell}\sum_{j\neq j'}^n\mathbb{E}[\phi_{ij}^2\phi_{i'j'}^2]
\\
=~&\sum_{i=1}^{q\ell}\sum_{j,j'\in\mathcal{L}_i, j\neq j'}\mathbb{E}[\phi_{ij}^2\phi_{ij'}^2] + \sum_{i\neq i'}^{q\ell}\sum_{j\in\mathcal{L}_i, j'\in\mathcal{L}_{i'}, j\neq j'}\mathbb{E}[\phi_{ij}^2\phi_{i'j'}^2]
\\
=~& q\ell \cdot\dfrac{n}{q} (\dfrac{n}{q} -1)\cdot\dfrac{1}{\ell^2} +
\left( \sum_{i\neq i', \mathcal{L}_i\neq\mathcal{L}_{i'}} \sum_{j\in\mathcal{L}_i, j'\in\mathcal{L}_{i'}}\mathbb{E}[\phi_{ij}^2\phi_{i'j'}^2] \right. \\
~&  \left. + \sum_{i\neq i', \mathcal{L}_i=\mathcal{L}_{i'}}\sum_{j,j'\in\mathcal{L}_i, j\neq j'}\mathbb{E}[\phi_{ij}^2\phi_{i'j'}^2] \right)
\\
=~& \dfrac{n^2-nq}{q\ell}  + \left( q(q -1) \ell^2\dfrac{n^2}{q^2}\dfrac{1}{\ell^2}+  q\ell(\ell-1)\dfrac{n}{q} \left(\dfrac{n}{q}-1\right)\dfrac{1}{\ell^2} \right)
\\
=~& n^2-n,
\end{align*}
where the second equality holds since $\phi_{ij}^2\phi_{i'j'}^2 = 0$ if $j\notin \mathcal{L}_i$ or $j'\notin\mathcal{L}_{i'}$; the
third equality is by partitioning $i\neq i'$ into $\mathcal{L}_i \neq \mathcal{L}_{i'}$ and $\mathcal{L}_i = \mathcal{L}_{i'}$. Note there are $q$ different $\mathcal{L}_i$'s and $\ell$ choices of $i$ fall in the same $ \mathcal{L}_i$; for each $i$, as long as $j\in\mathcal{L}_i$, then $\mathbb{E}[\phi_{ij}^2] = \dfrac{1}{\ell}$, and for $j\neq j'$, $\phi_{ij}$ and $\phi_{i'j'}$ are independent, then the fourth equality follows. Therefore,
\begin{align}
\label{eq: app2}
 ~& \sum_{i,i'=1}^{q\ell}\sum_{j\neq j'}^n\mathbb{E}[\phi_{ij}^2\phi_{i'j'}^2]\mathbb{E}[(PU)_{js}(PU)_{jt}(PU)_{j's}(PU)_{j't}] \nonumber\\
 =~& \dfrac{1}{n(n-1)}\sum_{j\neq j'}U_{js}U_{jt}U_{j's}U_{j't} \cdot \left(\sum_{i,i'=1}^{q\ell}\sum_{j\neq j'}^n\mathbb{E}[\phi_{ij}^2\phi_{i'j'}^2] \right) \nonumber\\
=~& \sum_{j\neq j'}U_{js}U_{jt}U_{j's}U_{j't}
\end{align}
\textbf{3.} $j_1 = j_1' \neq j_2 = j_2'$:
\[
\mathbb{E}[(PU)_{j_1s}^2(PU)_{j_2t}^2] = \dfrac{1}{n(n-1)}\sum_{j_1\neq j_2}U_{j_1s}^2U_{j_2t}^2;
\]
and since $\phi$ has only one $1$ in each column, then for $i\neq i'$, at least one of $\phi_{ij_1}, \phi_{i'j_1}$ is $0$, and $\phi_{ij_1}\phi_{ij_2}\phi_{i'j_1}\phi_{i'j_2} = 0$, thus
\begin{align*}
~& \sum_{i,i'=1}^{q\ell}\sum_{j_1\neq j_2}^n\mathbb{E}[\phi_{ij_1}\phi_{ij_2}\phi_{i'j_1}\phi_{i'j_2}] \\
=~& \sum_{i=1}^{q\ell}\sum_{j_1\neq j_2}^n\mathbb{E}[\phi_{ij_1}^2\phi_{ij_2}^2] + \sum_{i\neq i'}^{q\ell}\sum_{j_1\neq j_2}^n\mathbb{E}[\phi_{ij_1}\phi_{ij_2}\phi_{i'j_1}\phi_{i'j_2}]  \\
=~&\sum_{i=1}^{q\ell}\sum_{j_1\neq j_2, ~ j_1, j_2\in\mathcal{L}_i}\mathbb{E}[\phi_{ij_1}^2\phi_{ij_2}^2] + 0\\
=~&q\ell \cdot\dfrac{n}{q} \left(\dfrac{n}{q}-1\right)\cdot \dfrac{1}{\ell^2}\\
=~& \dfrac{n}{\ell}\left(\dfrac{n}{q}-1\right).
\end{align*}
Therefore,
\begin{align}
\label{eq: app3}
& \sum_{i,i'=1}^{q\ell}\sum_{j_1\neq j_2}^n\mathbb{E}[\phi_{ij_1}\phi_{ij_2}\phi_{i'j_1}\phi_{i'j_2}] \mathbb{E}[(PU)_{j_1s}^2(PU)_{j_2t}^2] \nonumber\\
=& \dfrac{1}{n(n-1)} \sum_{j_1\neq j_2} U_{j_1s}^2U_{j_2t}^2\cdot \left( \sum_{i,i'=1}^{q\ell}\sum_{j_1\neq j_2}^n\mathbb{E}[\phi_{ij_1}\phi_{ij_2}\phi_{i'j_1}\phi_{i'j_2}] \right) \nonumber\\
=& \dfrac{n/q-1}{\ell(n-1)}\sum_{j_1\neq j_2}U_{j_1s}^2U_{j_2t}^2.
\end{align}
\textbf{4.} $j_1 = j_2' \neq j_1' = j_2$:
\begin{align*}
~& \mathbb{E}[(PU)_{js}(PU)_{j't}(PU)_{j's}(PU)_{jt}] \\
=~& \dfrac{1}{n(n-1)}\sum_{j\neq j'}U_{js}U_{jt}U_{j's}U_{j't}~;
\\
~&\sum_{i,i'=1}^{q\ell}\sum_{j\neq j'}^n\mathbb{E}[\phi_{ij}\phi_{ij'}\phi_{i'j'}\phi_{i'j}] \\
=~& \sum_{i=1}^{q\ell}\sum_{j\neq j'}^n\mathbb{E}[\phi_{ij}^2\phi_{ij'}^2] +\sum_{i\neq i'}^n\sum_{j\neq j'}^n\mathbb{E}[\phi_{ij}\phi_{ij'}\phi_{i'j'}\phi_{i'j}]\\
=~& \sum_{i=1}^{q\ell}\sum_{j\neq j', ~ j,j'\in\mathcal{L}_i}\mathbb{E}[\phi_{ij}^2\phi_{ij'}^2] +0
\\
=~& q\ell \cdot \dfrac{n}{q} \left(\dfrac{n}{q}  -1\right) \cdot \dfrac{1}{\ell^2} \\
=~& \dfrac{n}{\ell}\left(\dfrac{n}{q}-1\right).
\end{align*}
Therefore,
\begin{eqnarray}
\label{eq: app4}
&& \sum_{i,i'=1}^{q\ell}\sum_{j\neq j'}^n  \mathbb{E}[\phi_{ij}\phi_{ij'}\phi_{i'j'}\phi_{i'j}] \nonumber\\
&& ~~~~~~~~~~~~~~\cdot\mathbb{E}[(PU)_{js}(PU)_{jt}(PU)_{j's}(PU)_{j't}]\nonumber\\
&=& \dfrac{n/q-1}{\ell(n-1)}\sum_{j\neq j'}U_{js}U_{jt}U_{j's}U_{j't}.
\end{eqnarray}
Now combining (\ref{eq: app1}), (\ref{eq: app2}), (\ref{eq: app3}) and (\ref{eq: app4}),  we obtain
\begin{align*}
~ & \mathbb{E}[x_{st}^2]\\
=~ & \sum_{j=1}^nU_{js}^2U_{jt}^2 + \sum_{j\neq j'}U_{js}U_{jt}U_{j's}U_{j't} \\
  ~ & + \dfrac{n/q-1}{(n-1)\ell}\sum_{j_1\neq j_2}U_{j_1s}^2U_{j_2t}^2  + \dfrac{n/q-1}{(n-1)\ell}\sum_{j\neq j'}U_{js}U_{jt}U_{j's}U_{j't}\\
=~ & \left(\sum_{j=1}^n U_{js}U_{jt}\right)^2 \\
 ~& + \dfrac{n/q-1}{(n-1)\ell}\left(\left(\sum_{j=1}^nU_{js}^2 \right)\left(\sum_{j=1}^nU_{jt}^2 \right)- \sum_{j=1}^nU_{js}^2U_{jt}^2  \right)\\
~& + \dfrac{n/q -1}{(n-1)\ell}\left(\left(\sum_{j=1}^nU_{js}U_{jt}\right)^2-\sum_{j=1}^n U_{js}^2U_{jt}^2\right)\\
~\leq& \left\{\begin{alignedat}{2}
&1+ \dfrac{2(n/q-1)}{(n-1)\ell}& ~~~~~~ & s=t\\
&\dfrac{n/q-1}{(n-1)\ell}& ~~~~~~ & s\neq t
\end{alignedat}\right.\\
~\leq& \left\{\begin{alignedat}{2}
&1+\dfrac{2}{q\ell}& ~~~~~~ & s=t\\
&\dfrac{1}{q\ell}& ~~~~~~ & s\neq t
\end{alignedat}\right.
\end{align*}
The first inequality holds since $\dfrac{n/q-1}{(n-1)\ell}\geq 0$; the last inequality follows since $\dfrac{n/q-1}{(n-1)\ell} = \dfrac{1}{q\ell}\dfrac{n-q}{n-1}\leq \dfrac{1}{q\ell}$ as $q\geq 1$.
Therefore, we have
\begin{align*}
\mathbb{E}[\norm{X-I}_F^2] =& \sum_{s,t =1}^k\mathbb{E}[(x_{st}-\delta_{st})^2]
=  \sum_{s,t=1}^k\left( \mathbb{E}[x_{st}^2]-\delta_{st}\right) \\
\leq& \sum_{s=1}^k \dfrac{2 }{q\ell} +\sum_{s\neq t}^k \dfrac{1}{q\ell}
= \dfrac{k^2+k}{q\ell}.
\end{align*}
Applying Markov's inequality,  we obtain
\begin{align*}
~& \mathbb{P}\left[\norm{X-I}_2\leq \varepsilon\right]
\geq~ \mathbb{P}\left[\norm{X-I}^2_F\leq \varepsilon^2\right]
\geq~ 1-\dfrac{k^2+k }{q\ell \varepsilon^2}.
\end{align*}
By taking
$q\ell\geq\dfrac{k^2+k}{\varepsilon^2\delta}
$, we conclude that
\[\mathbb{P}\left[\norm{(SPU)^T(SPU)-I}_2\leq\varepsilon\right]\geq 1-\delta.\]
\end{proof}

\begin{remark}
Lemma~\ref{lemma: SpFD1} is to find a condition for a SpFD matrix to qualify as a $(1\pm\varepsilon)$ $\ell_2$-subspace embedding for $U$. By setting $q=1$, $SPA\in\mathbb{R}^{\ell\times d}$ and no FD is required, SpFD becomes SpEmb. Then the bound derived here for SpFD agrees with that for a general SpEmb matrix (Theorem 2.5 in \cite{Woodruff2014}). 
\end{remark}

Next, we move on to discuss the FD part in our algorithm and to determine the bound on the final sketch size $\ell$. After obtaining $SPA$, the next step is applying FD on $SPA$ to generate the sketch $B$, the orthonormal basis $V$ and eventually the approximation $\tilde{A}_k$.

As mentioned in Section~\ref{sec: FD}, the FD process in Algorithm~\ref{alg: fd} and Algorithm~\ref{alg: SpFD} performs one last round of SVD and $B$ update at the end. Nevertheless, the same properties can be easily derived following the same procedure in \cite{Liberty2013, Ghashami2015}. Then for the FD process embedded in Algorithm~\ref{alg: SpFD} of SpFD, we have the following properties: let
\[\Delta = \sum_{i=1}^{q-1}\delta_i,
\]
Property 1: For any vector $x\in\mathbb{R}^d$, $\norm{SPAx}_2^2 - \norm{Bx}_2^2 \geq 0$.
\\
Property 2: For any unit vector $x\in\mathbb{R}^d$, $\norm{SPAx}_2^2- \norm{Bx}_2^2$ $\leq \Delta$.
\\
Property 3: $\norm{SPA}_F^2 - \norm{B}_F^2 \geq \ell\Delta$.

Besides Lemma~\ref{lemma: SpFD1} and the properties listed above, we need one more lemma to derive the main theorem.

\begin{lemma}
\label{lemma: SpFD2}
For any matrix $C\in\mathbb{R}^{n\times d}$, let $S\in\mathbb{R}^{q\ell\times n}$ be defined as (\ref{eq: S}) where $q$ and $\ell$ are positive integers and $P$ be an $n\times n$ random permutation matrix defined earlier,
then
$\mathbb{E}\left[ \norm{SPC}_F^2 \right] =\norm{C}_F^2$.
Furthermore, by applying Markov's inequality, with probability at least $1-\delta$ ($0<\delta<1$),
\[
\norm{SPC}_F^2\leq \dfrac{1}{\delta}\norm{C}_F^2.
\]
\end{lemma}
This lemma states that without any requirement on the size of the sketching matrix, the expected value of the $F$-norm of the sketched matrix is the $F$-norm of the original matrix. The proof is provided in Section 1 of the  supplemental material. Now we are ready to derive our main theorem.

\begin{theorem}
Given a matrix $A\in\mathbb{R}^{n\times d}$ and let $A_k$ denote the best rank-$k$ approximation of $A$.
For any $\varepsilon, \delta\in(0,1)$,  let $S\in\mathbb{R}^{q\ell\times n}$ be defined as (\ref{eq: S})
with
\[
q\ell \geq \dfrac{12(k^2+k)}{\varepsilon^2\delta} \text{~ and ~~} \ell = \left\lceil k + \dfrac{3k}{\varepsilon\delta}\right\rceil
\]
and $P$ be an $n\times n$ random permutation matrix
(i.e., for each $e_j$ with $j\in[n]$, $\mathbb{P}[Pe_j = e_i] = 1/n$ for $i\in[n]$).
Let $\tilde{A}_k$ and $B\in\mathbb{R}^{\ell\times d}$ be the output of Algorithm~\ref{alg: SpFD},
then with probability at least $1-\delta$,
\[
\norm{A-\tilde{A}_k}_F^2 = \norm{A-\Pi_{B, k}^F(A)}_F^2\leq \norm{A-A_k}_F^2 + \varepsilon\norm{A}_F^2.
\label{eq: error}
\]
\label{thm: main}
\end{theorem}
\begin{proof}
Let
\begin{eqnarray*}
A &=& W\Lambda Q^T = \begin{bmatrix} W_k & W_{\rho-k}\end{bmatrix} \begin{bmatrix}\Lambda_k & \\ & \Lambda_{\rho-k}\end{bmatrix}\begin{bmatrix} Q_k^T & Q_{\rho-k}^T\end{bmatrix}\\
B &=& U\Sigma V^T = \begin{bmatrix} U_k & U_{\beta-k}\end{bmatrix} \begin{bmatrix}\Sigma_k & \\ & \Sigma_{\beta-k}\end{bmatrix}\begin{bmatrix} V_k^T & V_{\beta-k}^T\end{bmatrix}
\end{eqnarray*}
 be the singular value decompositions of $A\in\mathbb{R}^{n\times d} $ and $B\in\mathbb{R}^{\ell\times d}$, $\rho$ = rank($A$) and $\beta$ = rank($B$) and the singular values are listed in a non-increasing order in $\Lambda$ and $\Sigma$. Thus,
\begin{eqnarray}
\norm{A-\Pi_{B,k}^F(A)}_F^2 &=& \min_{X\in\mathbb{R}^{n\times \ell},~ rank(X)\leq k}\norm{A-XB}_F^2\nonumber\\
		&\leq& \norm{A-AV_kV_k^T}_F^2 \nonumber \\
&=& \norm{A-A_k}_F^2 + \norm{A_k}_F^2 - \norm{AV_k}_F^2.
\label{eq: SFD1}
\end{eqnarray}
The inequality holds by taking $X = AV_k\Sigma_k^{-1}U_k^T$; the last equality holds since $(A-AV_kV_k^T)V_kV_k^TA^T = 0$ which implies $\norm{A-AV_kV_k^T}_F^2 = \norm{A}_F^2-\norm{AV_k}_F^2$ (by Lemma 2.1 in \cite{Boutsidis2011}).

Let $U_{AV_k}\Sigma_{AV_k}V_{AV_k}$ be the SVD of $AV_k$ where $U_{AV_k}\in\mathbb{R}^{n\times k}$ is a matrix whose columns composed of the left singular vectors of $AV_k$; assume
\begin{equation}
q\ell\geq \dfrac{k^2+k}{\varepsilon_0^2\delta_0},
\label{eq: nalpha}
\end{equation}
for some $\varepsilon_0, \delta_0\in (0,1)$, then by Lemma~\ref{lemma: SpFD1}, with probability at least $1-\delta_0$, all singular values of $SPU_{AV_k}$ are in $[\sqrt{1-\varepsilon_0}, \sqrt{1+\varepsilon_0}]$. Then with probability at least $1-\delta_0$,
\begin{align*}
	\norm{SPAV_k}_F^2 -\norm{AV_k}_F^2
	\leq~
	& \left(\norm{SPU_{AV_k}}_2^2 -1\right)\norm{AV_k}_F^2\\
	\leq~& \varepsilon_0\norm{AV_k}_F^2,
\end{align*}
which implies
\begin{equation}
\norm{AV_k}_F^2 \geq \dfrac{1}{1+\varepsilon_0}\norm{SPAV_k}_F^2.
\label{eq: SFD2}
\end{equation}
Conditioned on (\ref{eq: SFD2}), then (\ref{eq: SFD1}) becomes
\begin{align}
	&\norm{A-\Pi_{B,k}^F(A)}_F^2\nonumber\\
	\leq& \norm{A-A_k}_F^2 + \norm{A_k}_F^2 -\dfrac{1}{1+\varepsilon_0}\norm{SPAV_k}_F^2\nonumber\\
	\leq& \norm{A-A_k}_F^2 + \norm{AQ_k}_F^2 -\dfrac{1}{1+\varepsilon_0}\norm{BV_k}_F^2\nonumber\\
	\leq& \norm{A-A_k}_F^2 + \norm{AQ_k}_F^2 -\dfrac{1}{1+\varepsilon_0}\norm{BQ_k}_F^2\nonumber\\
	=& \dfrac{1}{1+\varepsilon_0}\left(\norm{AQ_k}_F^2 - \norm{BQ_k}_F^2\right) + \dfrac{\varepsilon_0}{1+\varepsilon_0}\norm{AQ_k}_F^2 \nonumber\\
	& + \norm{A-A_k}_F^2
\label{eq: SFD3}
\end{align}
The second inequality is based on $\norm{A_k}_F^2 = \norm{AQ_k}_F^2$ and (Property 1). The third inequality follows since $V_k$ contains the top $k$ right singular vectors of $B$ and thus $\norm{BV_k}_F^2 \geq \norm{BQ_k}_F^2$.

To bound $\norm{AQ_k}_F^2 - \norm{BQ_k}_F^2$, we use a bridge term $\norm{SPAQ_k}_F^2$ and bound the singular values of $SPW_k$ using Lemma~\ref{lemma: SpFD1} again. Still assume (\ref{eq: nalpha}),
thus with probability at least $1-\delta_0$, for all $1 \leq i\leq k$,
\begin{align*}
\sqrt{1-\varepsilon_0}\leq~& \sigma_{i}(SPW_k)\leq \sqrt{1+\varepsilon_0},
\\
\dfrac{1}{\sqrt{1+\varepsilon_0}}\leq~& \sigma_i((SPW_k)^\dagger) \leq \dfrac{1}{\sqrt{1-\varepsilon_0}}.
\end{align*}
This implies that $SPW_k\in\mathbb{R}^{q\ell \times k}$ is of full column rank and $(SPW_k)^\dagger SPW_k = I_k$, then
\begin{align*}
\norm{AQ_k}_F^2 =~& \norm{\Lambda_k}_F^2 =  \norm{(SPW_k)^\dagger SPW_k\Lambda_k}_F^2\\
				\leq~& \norm{(SPW_k)^\dagger}_2^2\norm{SPW_k\Lambda_k}_F^2 \\
				\leq~&  \dfrac{1}{1-\varepsilon_0}\norm{SPAQ_k}_F^2;\\
\norm{SPAQ_k}_F^2 \leq~& \norm{SPW_k}_2^2\norm{\Lambda_k}_F^2
 			\leq~ (1+\varepsilon_0)\norm{AQ_k}_F^2.
\end{align*}
Together they imply with probability at least $1-\delta_0$,
\begin{equation}
 \dfrac{1}{1+\varepsilon_0}\norm{SPAQ_k}_F^2  \leq \norm{AQ_k}_F^2 \leq \dfrac{1}{1-\varepsilon_0}\norm{SPAQ_k}_F^2.
\label{eq: SFD4}
\end{equation}
Conditioned on (\ref{eq: SFD2}) and (\ref{eq: SFD4}), then (\ref{eq: SFD3}) continues
\begin{align}
& \norm{A-\Pi_{B,k}^F(A)}_F^2\nonumber\\
	\leq~ &\dfrac{1}{1+\varepsilon_0}\left(\varepsilon_0\norm{AQ_k}_F^2+ (1-\varepsilon_0)\norm{AQ_k}_F^2 -\norm{BQ_k}_F^2 \right) \nonumber\\
	& + \dfrac{\varepsilon_0}{1+\varepsilon_0}\norm{AQ_k}_F^2 + \norm{A-A_k}_F^2\nonumber\\
	\leq~ &\dfrac{1}{1+\varepsilon_0}\left(\varepsilon_0\norm{AQ_k}_F^2 + \norm{SPAQ_k}_F^2-\norm{BQ_k}_F^2\right)\nonumber\\
	& + \dfrac{\varepsilon_0}{1+\varepsilon_0}\norm{AQ_k}_F^2 + \norm{A-A_k}_F^2\nonumber\\
	=~ & \dfrac{1}{1+\varepsilon_0}\left(\norm{SPAQ_k}_F^2-\norm{BQ_k}_F^2\right)\nonumber\\
	& + \dfrac{2\varepsilon_0}{1+\varepsilon_0}\norm{AQ_k}_F^2 + \norm{A-A_k}_F^2\nonumber\\
	\leq~ &\dfrac{1}{1+\varepsilon_0}\cdot k\Delta + \dfrac{2\varepsilon_0}{1+\varepsilon_0}\norm{A_k}_F^2 +\norm{A-A_k}_F^2.
\label{eq: SFD5}
\end{align}
The last inequality follows from (Property 2). To bound $\Delta$, we consider
\begin{align}
		~& \norm{A}_F^2 - \norm{B}_F^2\nonumber\\
 		=~& \norm{A}_F^2 - \norm{SPAQ_k}_F^2 +\norm{SPAQ_k}_F^2 -\norm{B}_F^2\nonumber\\
		\leq~& \norm{A}_F^2 - \norm{SPAQ_k}_F^2 + \norm{SPAQ_k}_F^2-\norm{BQ_k}_F^2\nonumber\\
		\leq~& \norm{A}_F^2 -\norm{SPAQ_k}_F^2 + k\Delta,
\label{eq: Af1}
\end{align}
where the last inequality follows from (Property 2). On the other hand, recall that
$A = W_k\Lambda_kQ_k^T + W_{\rho-k}\Lambda_{\rho-k}Q_{\rho-k}^T$, then
\begin{align}
	&~\norm{A}_F^2 -\norm{B}_F^2\nonumber\\
 	=~& \norm{A}_F^2 - \norm{SPW_k\Lambda_kQ_k^T +SPW_{\rho-k}\Lambda_{\rho-k}Q_{\rho-k}^T}_F^2  \nonumber\\
 	& + \norm{SPA}_F^2 -\norm{B}_F^2 \nonumber\\
 	=~& \norm{A}_F^2 -\norm{SPAQ_k}_F^2 - \norm{SPW_{\rho-k}\Lambda_{\rho-k}}_F^2\nonumber\\
 	& + \norm{SPA}_F^2-\norm{B}_F^2\nonumber\\
 	\geq~&  \norm{A}_F^2 -\norm{SPAQ_k}_F^2 - \norm{SPW_{\rho-k}\Lambda_{\rho-k}}_F^2  + \ell\Delta\nonumber\\
 	\geq~&  \norm{A}_F^2 -\norm{SPAQ_k}_F^2 - \dfrac{1}{\delta_1}\norm{\Lambda_{\rho-k}}_F^2 + \ell\Delta \nonumber\\
 	=~&  \norm{A}_F^2 -\norm{SPAQ_k}_F^2 - \dfrac{1}{\delta_1}\norm{A-A_k}_F^2 + \ell\Delta
\label{eq: Af2}
\end{align}
The second equality follows from
\[
SPW_k\Lambda_kQ_k^TQ_{\rho-k}\Lambda_{\rho-k}W_{\rho-k}^T(SP)^T = 0
\] which implies
\begin{align*}
& \norm{SPW_k\Lambda_kQ_k^T +SPW_{\rho-k}\Lambda_{\rho-k}Q_{\rho-k}^T}_F^2\\
=~& \norm{SPW_k\Lambda_kQ_k^T}_F^2 +\norm{SPW_{\rho-k}\Lambda_{\rho-k}Q_{\rho-k}^T}_F^2
\end{align*}
again by Lemma 2.1 in \cite{Boutsidis2011}; and $
\norm{SPW_k\Lambda_kQ_k^T}_F^2 = \norm{SPAQ_kQ_k^T}_F^2 = \norm{SPAQ_k}_F^2$.
The first inequality is based on (Property 3). And the last inequality
\begin{equation}
\norm{SPW_{\rho-k}\Lambda_{\rho-k}}_F^2 \leq \dfrac{1}{\delta_1}\norm{\Lambda_{\rho-k}}_F^2
\label{eq: SpFDExp}
\end{equation}
holds with probability at least $1-\delta_1$ where $\delta_1\in (0, 1)$, followed from Lemma~\ref{lemma: SpFD2} by setting $\delta=\delta_1$ and $C=W_{\rho-k}\Lambda_{\rho-k}$.
Conditioned on (\ref{eq: SpFDExp}), we combine (\ref{eq: Af1}) and (\ref{eq: Af2}), then
\begin{equation}
\Delta\leq \dfrac{1}{(\ell-k)\delta_1}\norm{A-A_k}_F^2.
\label{eq: SFD7}
\end{equation}
Now we substitute (\ref{eq: SFD7}) to (\ref{eq: SFD5}) and apply union bound on (\ref{eq: SFD2}),  (\ref{eq: SFD4}) and (\ref{eq: SpFDExp}), thus we conclude that with probability at least $1-2\delta_0-\delta_1$,
\begin{align*}
& \norm{A-\Pi_{B,k}^F(A)}_F^2 \\
	\leq~& \dfrac{k}{(1+\varepsilon_0)(\ell-k)\delta_1}\norm{A-A_k}_F^2 +\dfrac{2\varepsilon_0}{1+\varepsilon_0}\norm{A_k}_F^2 \\
	& + \norm{A-A_k}_F^2\\
	\leq~& \left(1+\dfrac{k}{(\ell-k)\delta_1}\right)\norm{A-A_k}_F^2 + 2\varepsilon_0\norm{A_k}_F^2.
\end{align*}
By (\ref{eq: Fproof}),
\begin{align*}
	\norm{A-\tilde{A}_k}_F^2 =~& \norm{A-[AV]_kV^T}_F^2=~ \norm{A-\Pi_{B,k}^F(A)}_F^2\\
	 \leq~&  \left(1+\dfrac{k}{(\ell-k)\delta_1}\right)\norm{A-A_k}_F^2 + 2\varepsilon_0\norm{A_k}_F^2.
\end{align*}
Taking $\varepsilon_0 = \dfrac{1}{2}\varepsilon$, $\delta_0 = \delta_1 = \dfrac{1}{3}\delta$ and $\ell = \left\lceil k+ \dfrac{3k}{\varepsilon\delta}\right\rceil$, then
\begin{align*}
\norm{A-\tilde{A}_k}_F^2 =~& \norm{A-\Pi_{B,k}^F(A)}_F^2\\
				\leq~& (1+\varepsilon)\norm{A-A_k}_F^2 +\varepsilon\norm{A_k}_F^2.
\end{align*}
Hence by taking
\[
q\ell \geq \dfrac{12(k^2+k)}{\varepsilon^2\delta} = \dfrac{k^2+k}{(\varepsilon/2)^2\delta/3} \text{~ and ~~}
\ell = \left\lceil k+ \dfrac{3k}{\varepsilon\delta}\right\rceil,
\]
 we conclude that with probability at least $1-\delta$,
\begin{align*}
\norm{A-\tilde{A}_k}_F^2 =~& \norm{A-\Pi_{B,k}^F(A)}_F^2
				\leq~ \norm{A-A_k}_F^2 + \varepsilon\norm{A}_F^2.
\end{align*}
\end{proof}

Next, we analyze the computational complexity of Algorithm~\ref{alg: SpFD} for SpFD, comparing with Algorithm~\ref{alg: proto} for SpEmb and Algorithm~\ref{alg: fd} for FD. The analysis is in most detail as it will be used for explanation of the numerical results in Section~\ref{sec: 4} and \ref{sec: 5}. With the same sketch size $\ell$, the number of flops for each algorithm are summarized in TABLE~\ref{tab: flopcounts}.
\begin{table*}[!]
\centering
\footnotesize
\begin{tabular}{|l|c|c|c|}
\hline
		& SpEmb (Algorithm~\ref{alg: proto}) & FD (Algorithm~\ref{alg: fd}) & SpFD (Algorithm~\ref{alg: SpFD}) ($q>1$)\\
\hline
1. Form the sketch $SA$ or $B$ & $2\text{nnz}(A)$  & \multirow{3}{*}{$\left[6d(2\ell)^2+20(2\ell)^3\right](\dfrac{n}{\ell}-1)$} & \multirow{2}{*}{$2\text{nnz}(A)$}
\\
\cline{1-2}
2. Construct $V$  & \multirow{2}{*}{$4d\ell^2 - \dfrac{4}{3}\ell^3$} & & \multirow{2}{*}{$+\left[6d(2\ell)^2+20(2\ell)^3\right](q-1)$} \\
~~~~(through QR for SpEmb) & & &  \\
\hline
3.  Compute $[AV]_k$ (via SVD) & $2\text{nnz}(A)\ell +6n\ell^2+20\ell^3$ & $2\text{nnz}(A)\ell +6n\ell^2+20\ell^3$ & $2\text{nnz}(A)\ell +6n\ell^2+20\ell^3$ \\
 ~~~~and obtain $[AV]_kV^T$  & $+2n\ell k +2nd\ell $  & $+2n\ell k +2nd\ell $ & $+2n\ell k +2nd\ell $\\
\hline
\multirow{3}{*}{Total} & \multirow{2}{*}{$2nd\ell+2\text{nnz}(A)(\ell + 1) + 6n\ell^2$} & \multirow{2}{*}{$26nd\ell + 2\text{nnz}(A)\ell + 166n\ell^2$} & $2nd\ell+2\text{nnz}(A)(\ell+ 1) + 6n\ell^2$
\\
 & \multirow{2}{*}{$+2n\ell k + (4d+18\frac{2}{3}\ell)\ell^2$} &  \multirow{2}{*}{$+2n\ell k - (24d+140\ell)\ell^2$}  & $+2n\ell k + 24d(q-1)\ell^2$\\
 & & & $+ \left(160q-140\right)\ell^3$
 \\
 \hline
\end{tabular}
\caption{Flop counts with sketch size $\ell$}
\label{tab: flopcounts}
\end{table*}
Here note that QR factorization on a $d\times \ell$ $(d\geq \ell)$ matrix requires $4d\ell^2-\dfrac{4}{3}\ell^3$ flops; SVD on an $n\times \ell$ $(n\geq \ell)$ matrix requires $6n\ell^2 +20\ell^3$ flops and the flop counts follow from \cite{Golub2013}. For SpEmb, as the sketching matrix is generated on a probability distribution, we count the flops in worst case of constructing $SA$ \footnote{Note $SA$ is formed by combining the rows of $A$ directly and no explicit $S$ is created.}.
Then we can conclude:
\begin{itemize}
\item To achieve the best error bounds, SpFD requires $\ell$ to be linear on $k$ which is the same as for FD \cite{Liberty2013, Ghashami2015}, and lower compared to SpEmb \cite{Woodruff2014, Clarkson2013}.
\item Since $n$ is large, $n\geq d\geq \ell\geq k$, $n\gg \ell$ and $q\ell$ is only required to be $O(k^2)$ by Theorem~\ref{thm: main}, the number of iterations have been significantly reduced from $n/\ell$ to $q = q\ell/\ell$.
Thus SpFD runs much faster than FD for the same sketch size. Nevertheless,
when $q$ approaches $n/\ell$, $S\in\mathbb{R}^{q\ell\times n}$ becomes more square-like and contains many zero rows with a high probability,
so will $SPA$.  Hence performing FD on $SPA$ would still be much more efficient than on $A$. A detailed discussion on the number of zero rows in $S$ is provided in Section 2 of the supplemental material.
\end{itemize}

\section{Experimental Results}
\label{sec: 4}
In this section, we evaluate the performance of the proposed SpFD algorithm by comparing it with four other commonly used algorithms on both synthetic and real datasets. All algorithms are implemented in MATLAB R2016a, and all experiments are conducted on a PC Dell OptiPlex 7040.
The competing algorithms are:

1. Norm Sampling (NormSamp) \cite{Drineas2003, Holodnak2015, Drineas2006} : It is a sampling method. The sketch is composed of $\ell$ rows chosen i.i.d. from the $n$ rows of $A$ and scaled properly. Each row in the sketch takes $A_{(i)}/\sqrt{\ell p_i}$ with probability $p_i = \norm{A_{(i)}}_2^2/\norm{A}_F^2$. Norm sampling requires $O(nd)$ time to form the probability distribution $p_i$'s and construct the sketch \cite{Drineas2006}.

2. Discrete Cosine Transform (DCT) \cite{Avron2011, Urano2013}: It is a random projection method, a type of fast Fourier transform which has more practical advantages compared to SRHT \cite{Avron2011, Yang2016}. The sketch is formed as $\sqrt{n/\ell}RFDA\in\mathbb{R}^{\ell\times d}$ with
\begin{itemize}
\item $R\in\mathbb{R}^{\ell\times n}$ is a subset of $\ell$ rows from $I_n$, where the rows are chosen uniformly and without replacement;
\item $F\in\mathbb{R}^{n\times n}$ is a normalized discrete cosine transform matrix;
\item $D\in\mathbb{R}^{n\times n}$ is a random diagonal matrix whose entries are independently chosen to be $\pm1$ with equal probability.
\end{itemize}

3. SpEmb: Follow Algorithm~\ref{alg: proto}.

4. FD: Follow Algorithm~\ref{alg: fd}.

5. SpFD: Follow Algorithm~\ref{alg: SpFD}.
Here we employ three different $q$ for comparison, refer to Fig.~\ref{fig: intuition} for an easy understanding (in later discussions, the following three variations together are referred to as SpFDq):
\begin{itemize}
\item $q = 5$ (SpFD5): For $\alpha = 1, \hdots, 5$, each $S_\alpha\in\mathbb{R}^{\ell\times n/q}$ combines $n/5$ rows in $PA$ to $\ell$ rows in each iteration for SVD computation. In other words, $S\in\mathbb{R}^{q\ell\times n}$ consists of $5$ diagonal blocks, sketching $PA\in\mathbb{R}^{n\times d}$ into $5$ $\ell\times d$ blocks, then perform FD.
\item $q = 10$ (SpFD10): $S\in\mathbb{R}^{q\ell\times n}$ consists of $10$ diagonal blocks, sketching $PA\in\mathbb{R}^{n\times d}$ into $10$ $\ell\times d$ blocks, then perform FD.
\item $q = 50$ (SpFD50): $S\in\mathbb{R}^{q\ell\times n}$ consists of $50$ diagonal blocks, sketching $PA\in\mathbb{R}^{n\times d}$ into $50$ $\ell\times d$ blocks, then perform FD.
\end{itemize}
For randomized algorithms, after forming the sketch matrix $B$, we follow the same procedure as SpEmb, use QR factorization to construct $V$ and SVD to construct $[AV]_k$, and then obtain $\tilde{A}_k = [AV]_kV^T$.

The randomized algorithms with dense embedding matrices used in practice use sketch sizes of $\ell=k+p$, where
$p$ is a small nonnegative integer \cite{Halko2011}. It has also been demonstrated in \cite{Urano2013} that
$\ell=4k$ is empirically sufficient for sparse embeddings. These results have been taken into considerations while setting the sketch sizes in our numerical experiments below.

\subsection*{Synthetic Data} The construction of synthetic datasets follows from \cite{Liberty2013, Ghashami2015} which were used to test FD. Each row of the generated input matrix $A\in\mathbb{R}^{n\times d}$ consists of $k$ dimensional signal and $d$ dimensional noise ($k \ll d$). To put it in an accurate form, $A = S^cDU + N/\zeta$. $S^c\in\mathbb{R}^{n\times k}$ is the signal coefficient matrix and $S^c_{ij}\sim \mathcal{N}(0,1)$ i.i.d. for $i\in[n]$ and $j\in[k]$. $D\in\mathbb{R}^{k\times k}$ is a diagonal matrix with $D_{ii} = 1-(i-1)/m$ which gives linearly diminishing signal singular values. $U\in\mathbb{R}^{k\times d}$ is the signal row space matrix which contains a random $k$ dimensional subspace in $\mathbb{R}^d$ and $UU^T = I_k$. The matrix $S^cDU$ is of exactly rank-$k$ and constitutes the signal we wish to recover. $N\in\mathbb{R}^{n\times d}$ contributes additive Gaussian noise $N_{ij} \sim\mathcal{N}(0,1)$. According to \cite{Vershynin2011} and as mentioned in \cite{Liberty2013, Ghashami2015}, to ensure the signal is recoverable,
 it is required that $c_1\leq \zeta\leq c_2\sqrt{d/k}$ for constants $c_1$ and $c_2$ close to $1$. In the experiments, we consider $n=10000$ and $d = 1000$, $k\in\{10, 20, 50\}$ ($k=10$ as default value) and $\zeta \in \{5, 10, 15, 20\}$ ($\zeta=10$ as default value). We set the sketch size $\ell \in\{k:10: 200\}$. For each dataset setting, we generate $3$ matrices and for each of these matrices, we run each method $5$ times
and the results are taken as the median of all $15$ repetitions for each dataset setting.
\begin{table*}[!t]
\vspace{10pt}
\centering
\begin{tabular}{lccccc}
\hline
 & No. of            &       Data       &  nnz  &      Approximating & Sketch  \\
& data points ($n$)&dimension ($d$)&   \%  &    rank ($k$) & size ($\ell$)\\
\hline
w8a \cite{Platt1999} & 64700 & 300 & 3.88& 20& 20:10:150 \\
\hline
Birds \cite{Wah2011}& 11788 & 312 & 10.09 & 50 & 50:10:150\\
\hline
Protein \cite{Wang2002} &24387& 357 & 28.20 & 50 & 50:10:170 \\
\hline
MNIST-all \cite{LeCun1998}& 70000 & 784 & 19.14 & 100& 100:10:200 \\
\hline
amazon7-small \cite{Blondel2013, Dredze2008} & 262144 & 1500 &0.017 & 100& 100:20:300 \\
\hline
rcv1-small \cite{Lewis2004} & 47236 & 3000 & 0.14& 100& 100:20:300 \\
\hline
\end{tabular}
\begin{remark*}
Dataset \textnormal{amazon7-small} is obtained from \textnormal{amazon7} which contains $n = 1362109$ data (number of rows) and $d = 262144$ features (number of columns). We pick the first $1500$ rows from the original matrix and use the transpose.
Dataset \textnormal{rcv1-small} is obtained from \textnormal{rcv1-multiclass} which contains $n = 534135$ data (number of rows) and $d = 47236$ features (number of columns). We pick the first $3000$ rows from the original matrix and use the transpose.
\end{remark*}
\caption{Data statistics.}
\label{tab: data}
\end{table*}
\subsection*{Real Data} We consider the six real world datasets in TABLE~\ref{tab: data}.
 For each dataset, we run $10$ times for each method and the results presented are the median of the $10$ repetitions.
 
\subsection*{Results} To measure the accuracy, we set the best rank-$k$ approximation $A_k$ obtained through SVD of $A$ as our benchmark. For each method, we examine the following three measures:
\begin{itemize}
\item $F$-norm Error: \quad $\norm{A-\tilde{A}_k}_F/\norm{A-A_k}_F$;
\item $2$-norm Error: \quad $\norm{A-\tilde{A}_k}_2/\norm{A-A_k}_2$;
\item Running time to construct $\tilde{A}_k$ in seconds.
\end{itemize}

For synthetic datasets of different signal dimension $k$ with $k=10, 20, 50$ while $\zeta = 10$, and different noise ratio $\zeta$ with $\zeta = 5, 15, 20$ while $k=10$, the results are displayed in Fig.~\ref{fig: setSynthetic} (a) - (f). The results for real datasets are displayed in Fig.~\ref{fig: setSynthetic} (g) - (h) and Fig.~\ref{fig: setReal} (i) - (l).

\begin{figure*}[!]
\subfigure{\makebox[15pt][r]{\makebox[20pt]{\raisebox{60pt}{\rotatebox[origin=c]{90}{$F$-norm Error}}}}
\includegraphics[width=4.2cm,height=3.5cm]{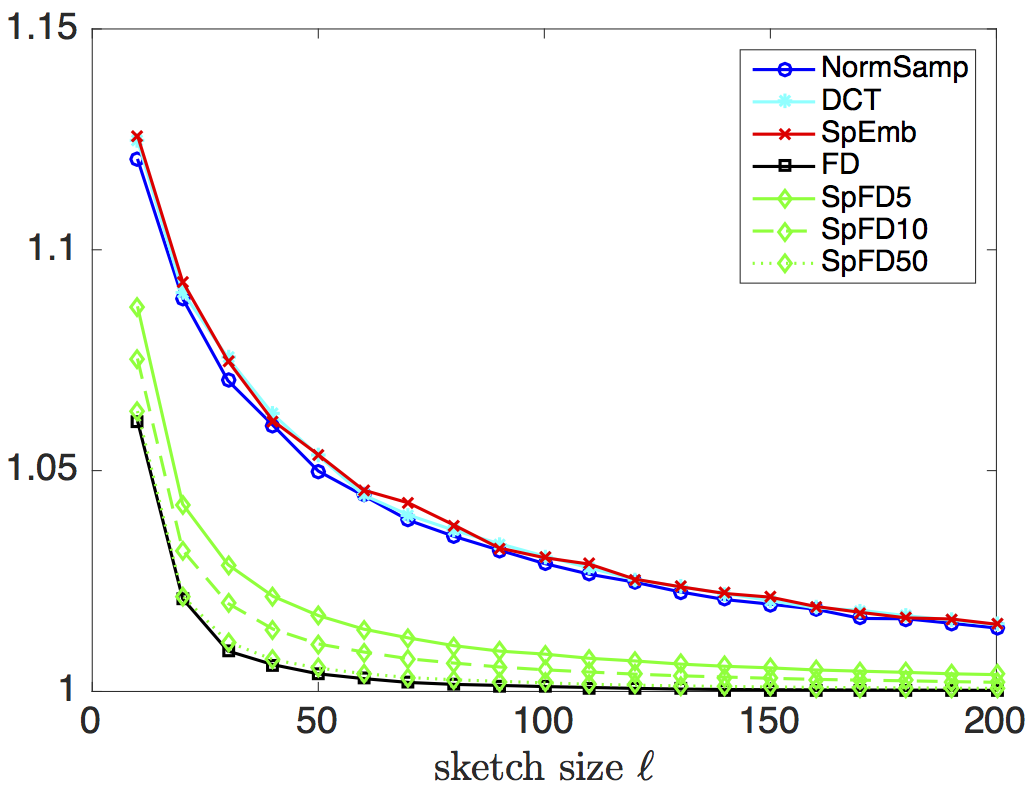}} \hspace{1pt}
\subfigure{\includegraphics[width=4.15cm,height=3.5cm]{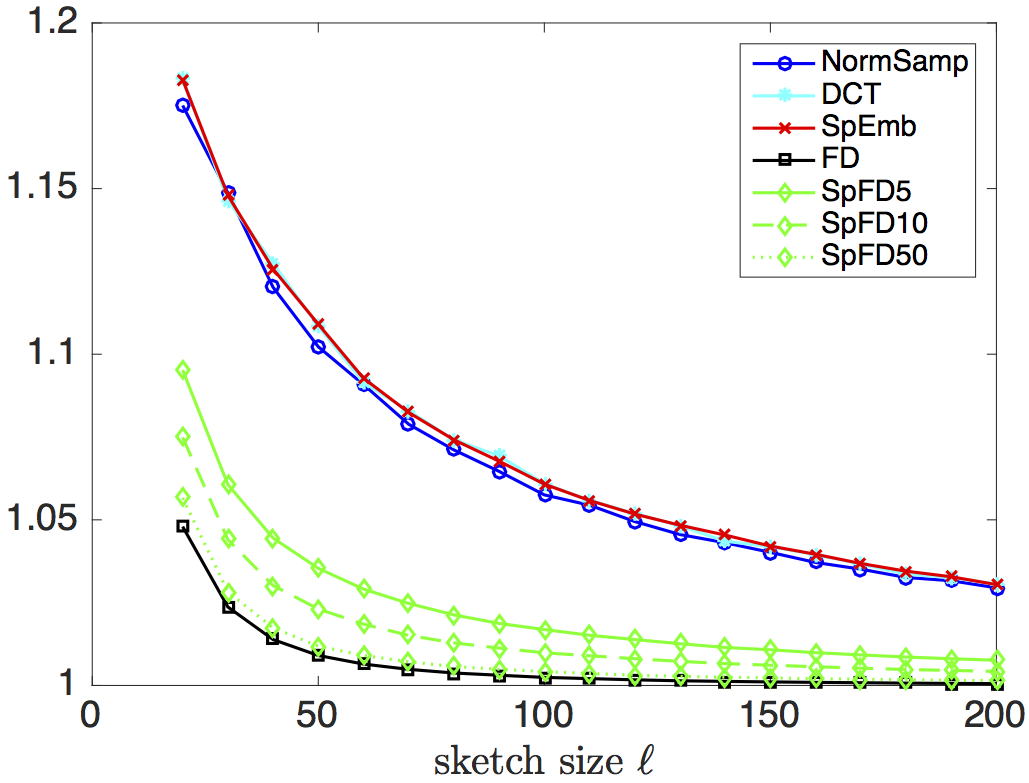}} \hspace{3pt}
\subfigure{\includegraphics[width=4.15cm,height=3.5cm]{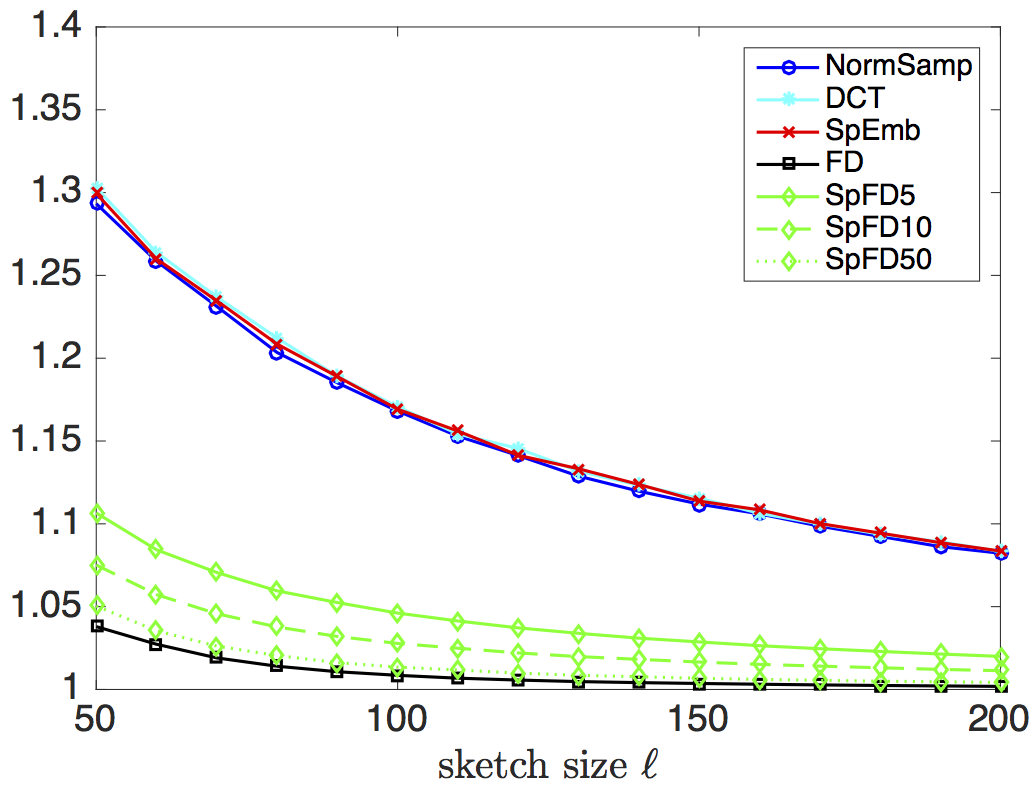}}\hspace{3pt}
\subfigure{\includegraphics[width=4.1cm,height=3.5cm]{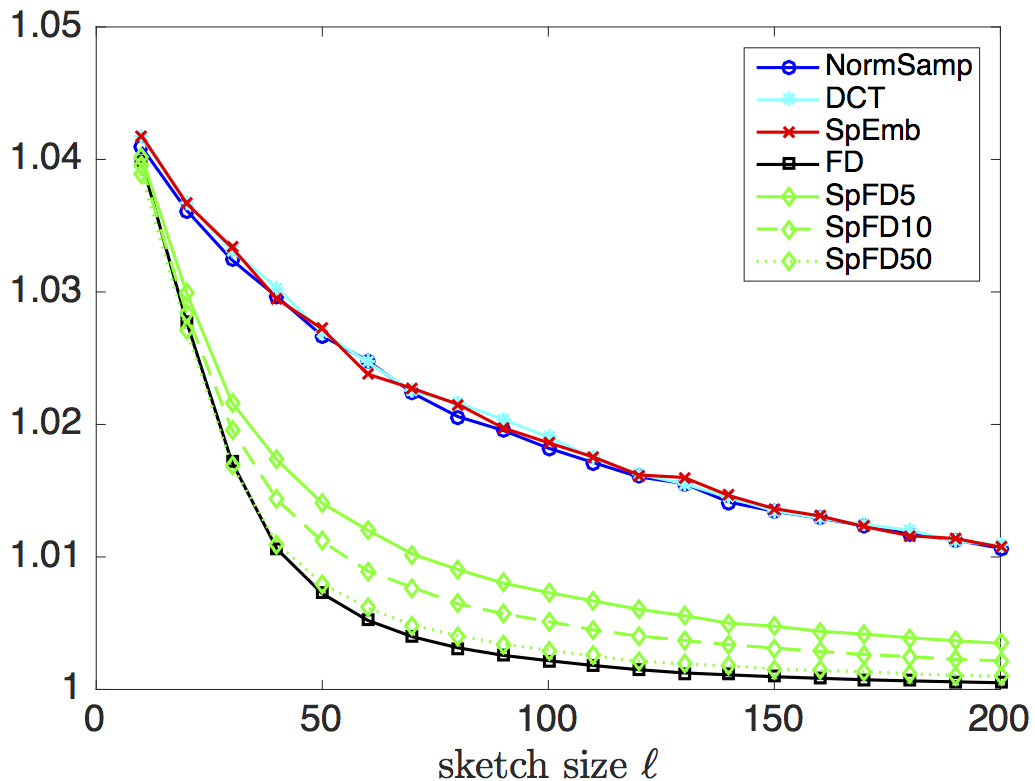}} \\
\vspace{-10pt}

\hspace{5pt}\subfigure{\makebox[15pt][r]{\makebox[20pt]{\raisebox{60pt}{\rotatebox[origin=c]{90}{$2$-norm Error}}}}
\includegraphics[width=4.0cm,height=3.5cm]{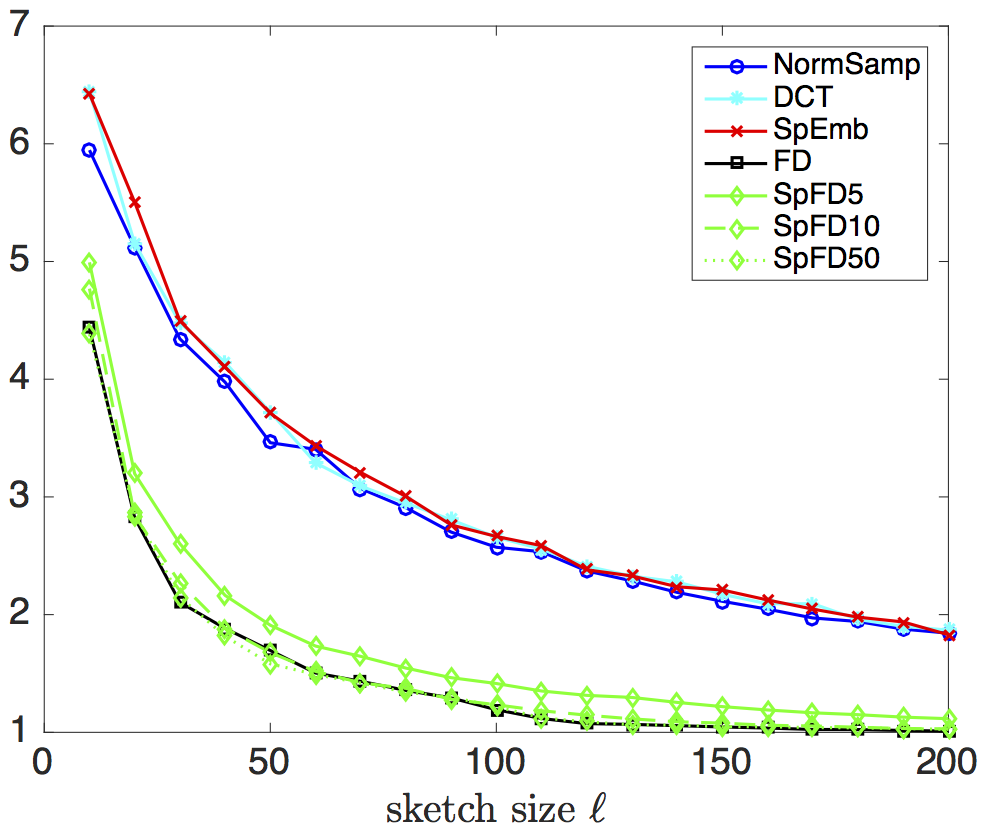}} \hspace{5pt}
\subfigure{\includegraphics[width=4.1cm,height=3.5cm]{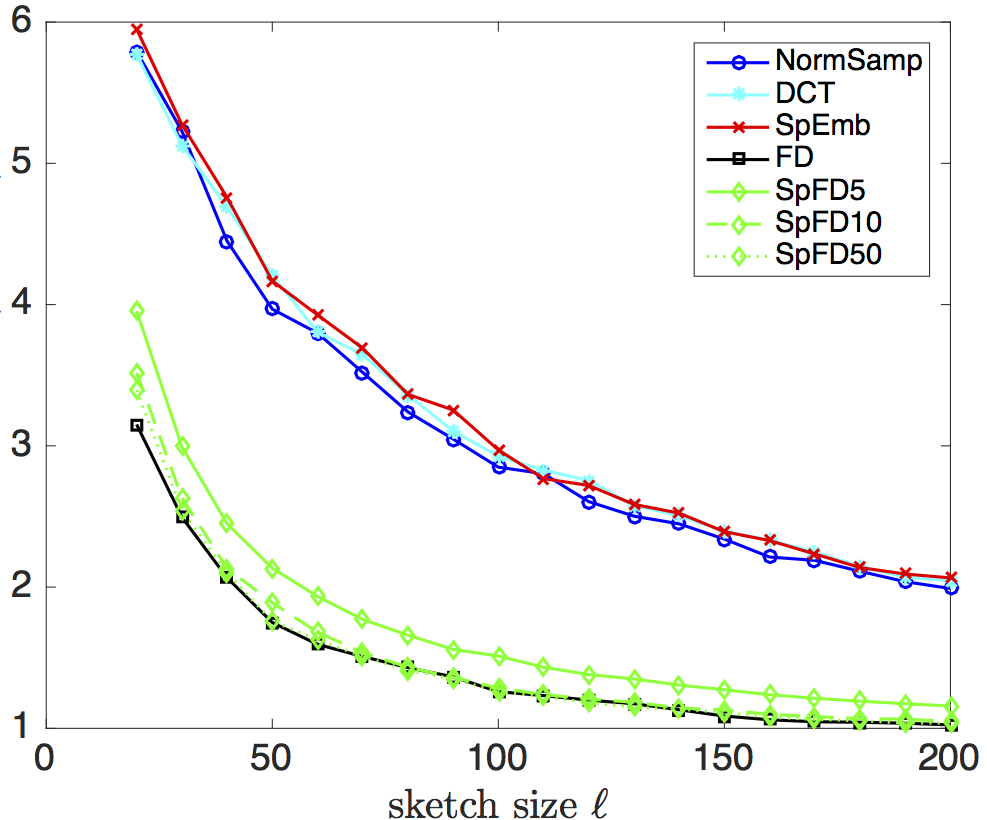}} \hspace{5pt}
\subfigure{\includegraphics[width=4.1cm,height=3.5cm]{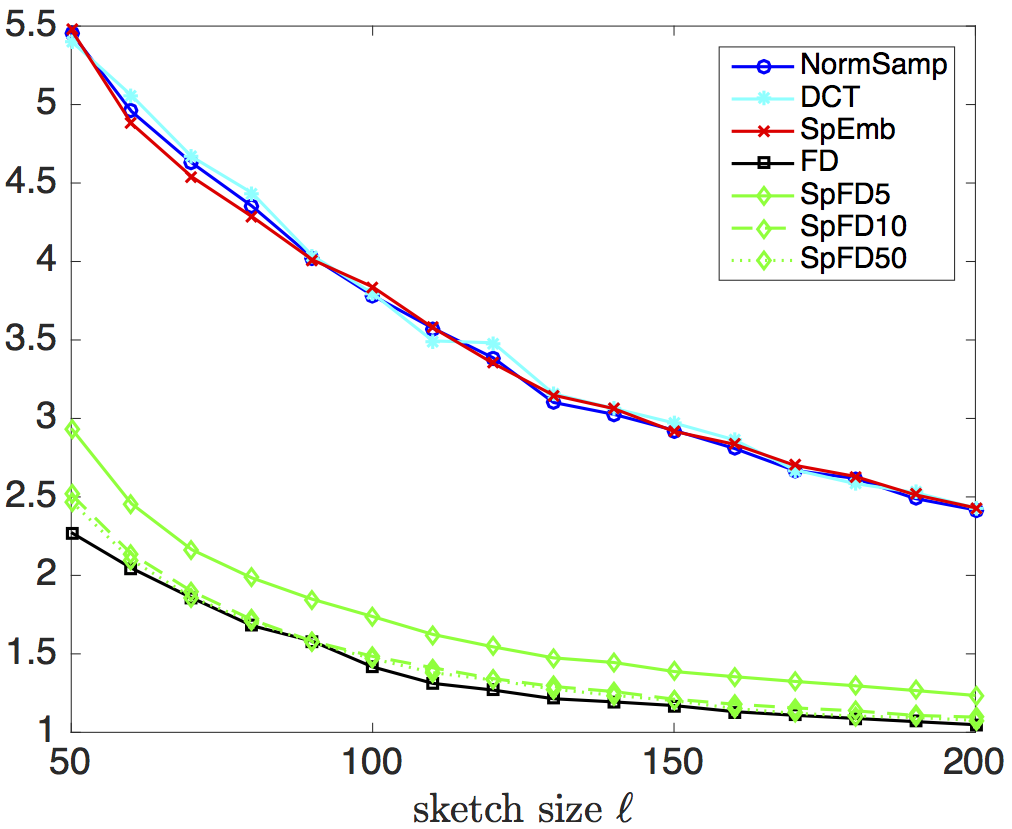}}\hspace{5pt}
\subfigure{\includegraphics[width=4.1cm,height=3.5cm]{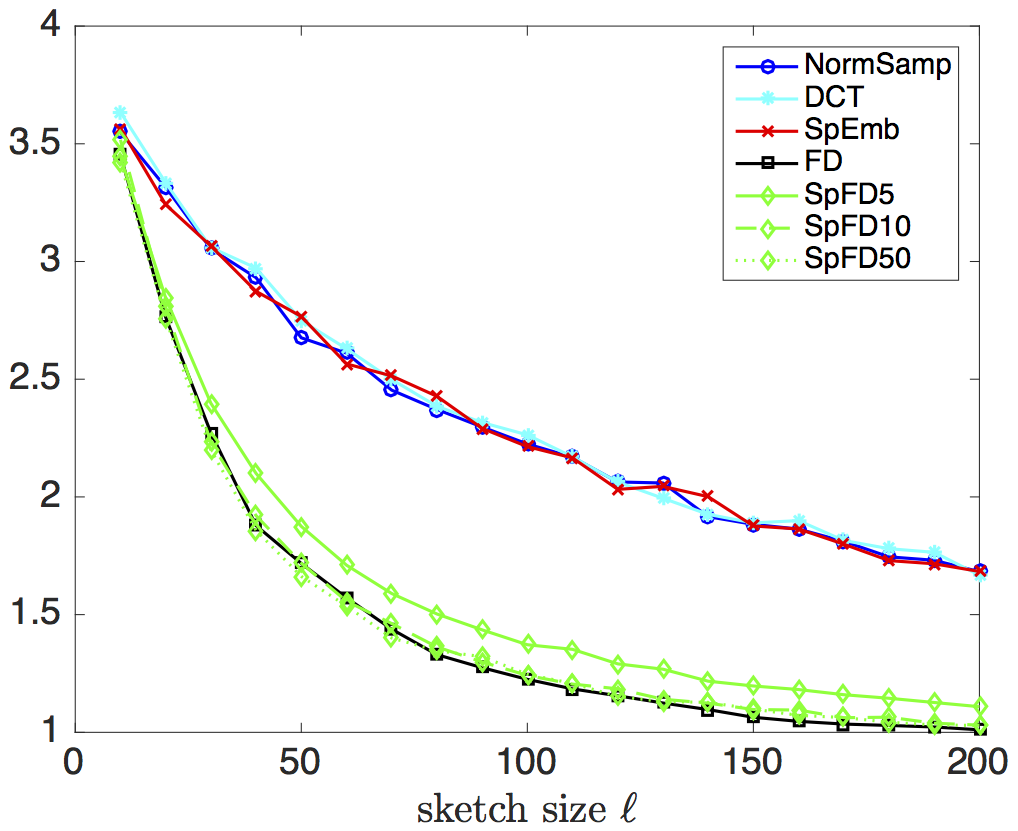}} \\
\vspace{-15pt}

\setcounter{subfigure}{0}
\subfigure[][$k=10, \zeta = 10$]{\makebox[15pt][r]{\makebox[20pt]{\raisebox{60pt}{\rotatebox[origin=c]{90}{Running Time (s)}}}}
\includegraphics[width=4.1cm,height=3.5cm]{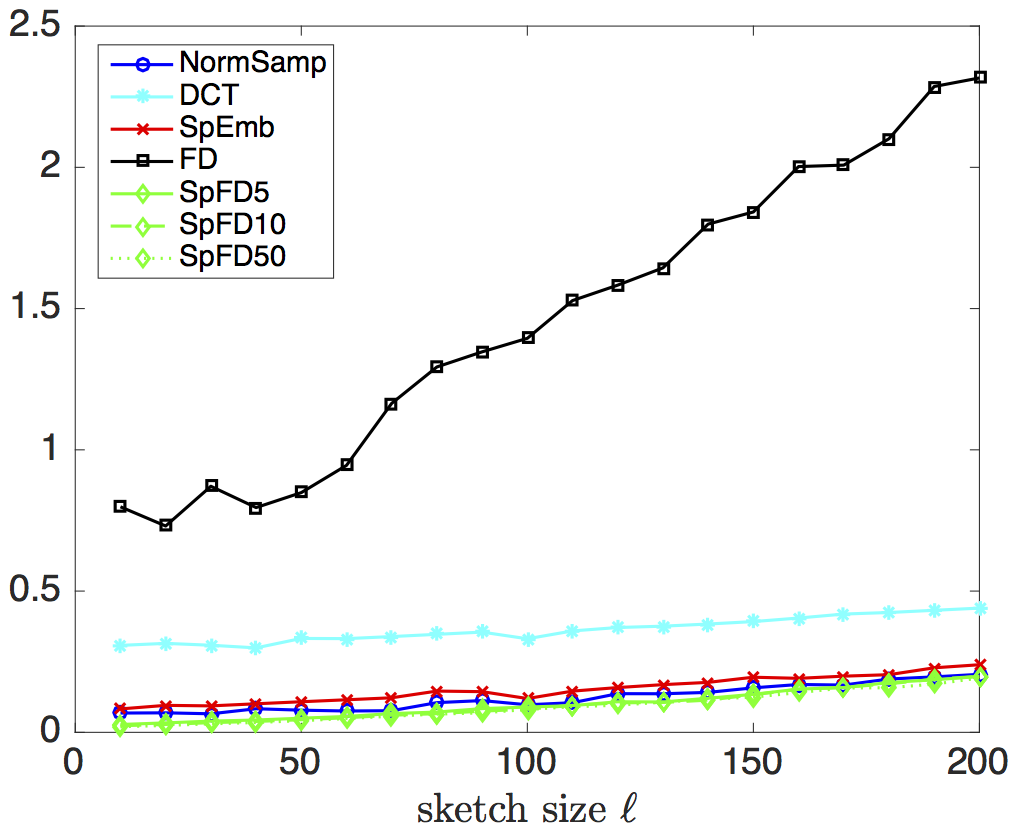}} \hspace{5pt}
\subfigure[][$k=20, \zeta = 10$]{\includegraphics[width=4.1cm,height=3.5cm]{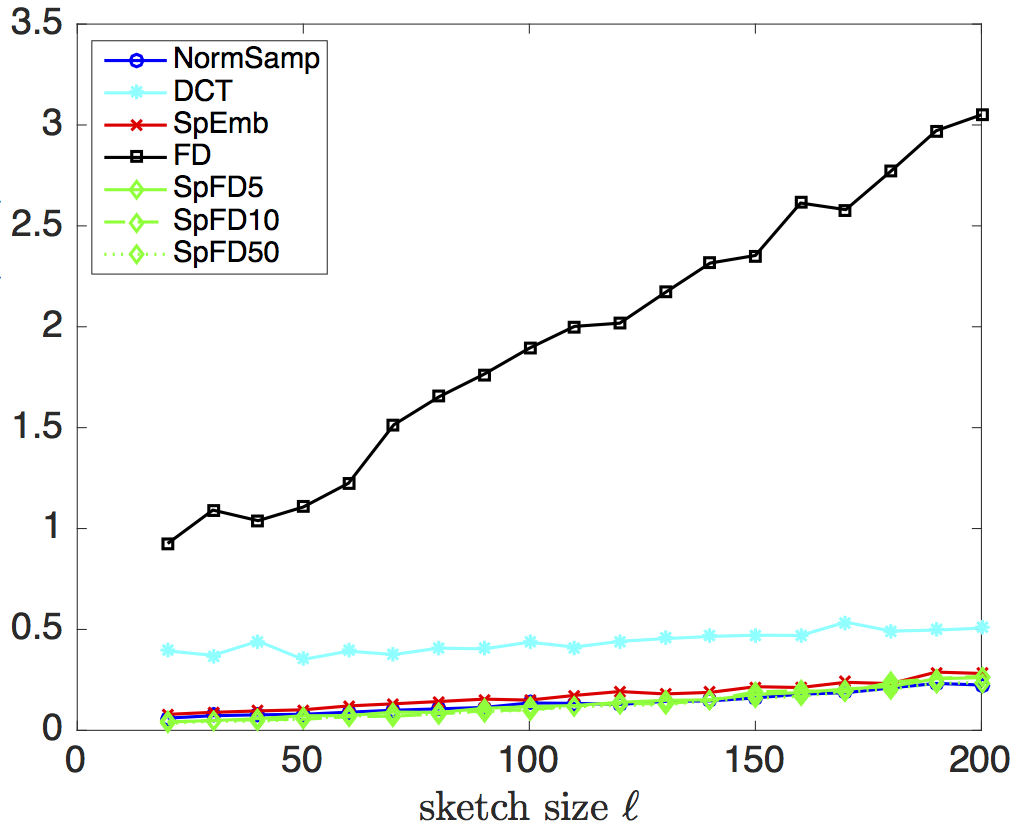}} \hspace{5pt}
\subfigure[][$k=50, \zeta = 10$]{\includegraphics[width=4.1cm,height=3.5cm]{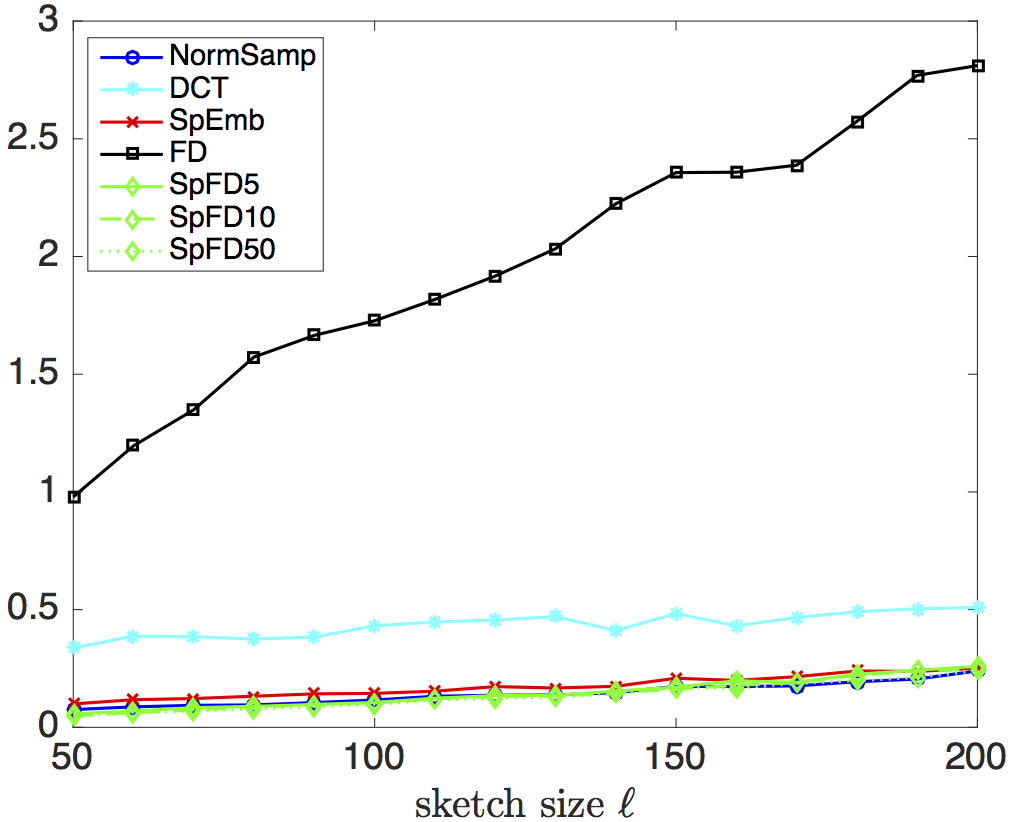}} \hspace{5pt}
\subfigure[][$k=10, \zeta = 5$]{\includegraphics[width=4.1cm,height=3.5cm]{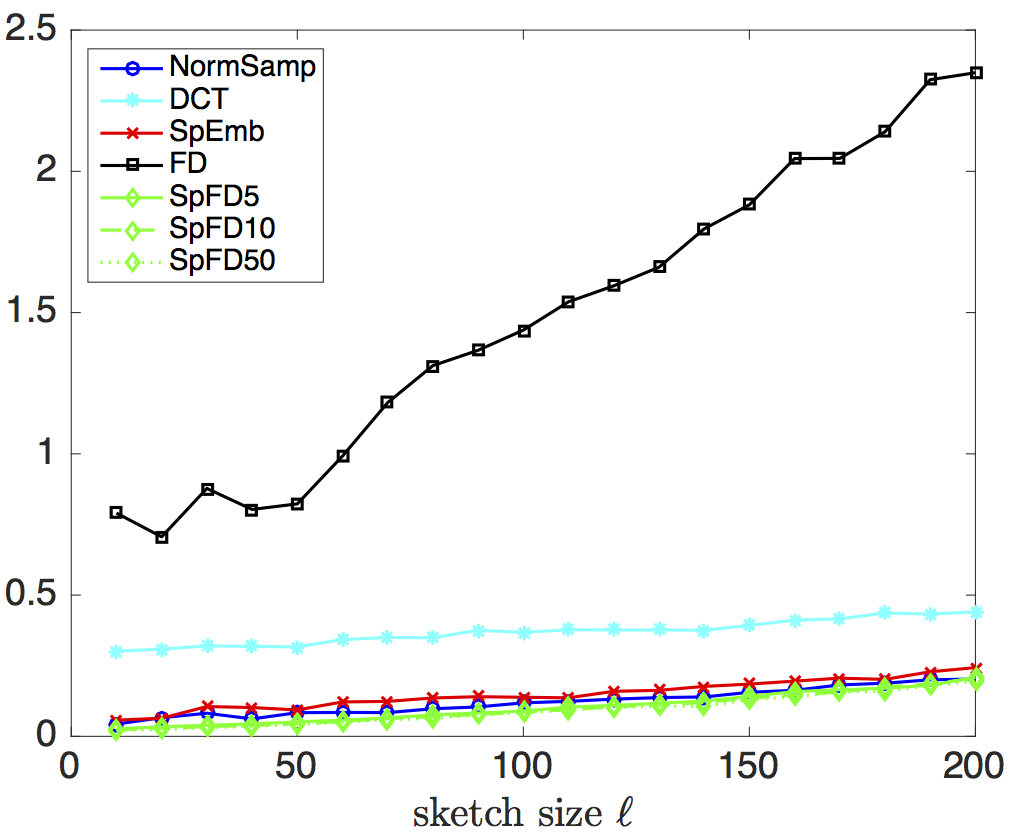}} \\

\subfigure{\makebox[15pt][r]{\makebox[20pt]{\raisebox{60pt}{\rotatebox[origin=c]{90}{$F$-norm Error}}}}%
\includegraphics[width=4.3cm,height=3.5cm]{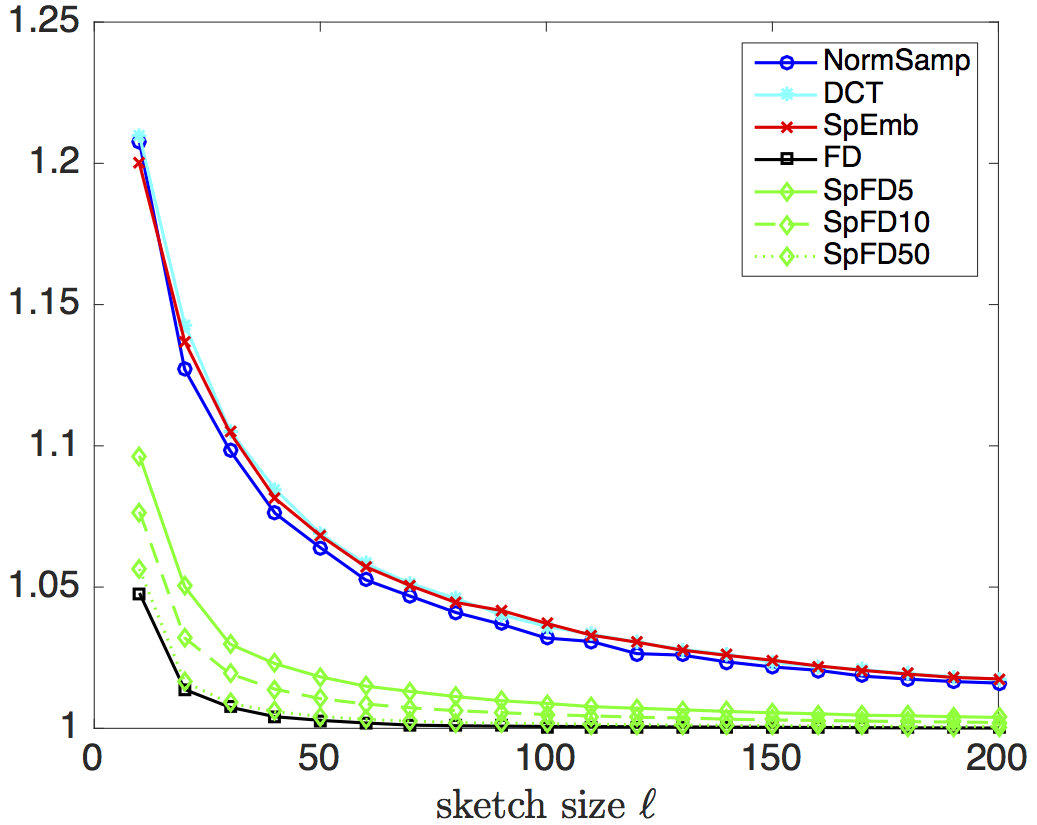}} \hspace{2pt}
\subfigure{\includegraphics[width=4.2cm,height=3.5cm]{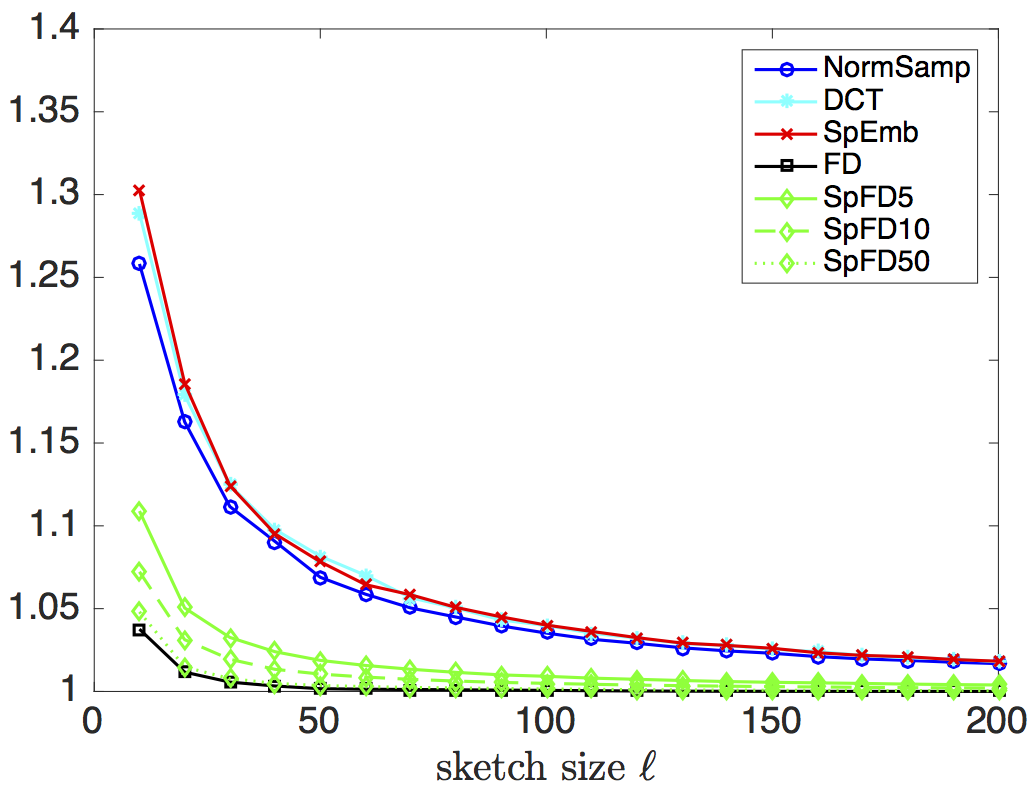}} \hspace{4pt}
\subfigure{\includegraphics[width=4.2cm,height=3.5cm]{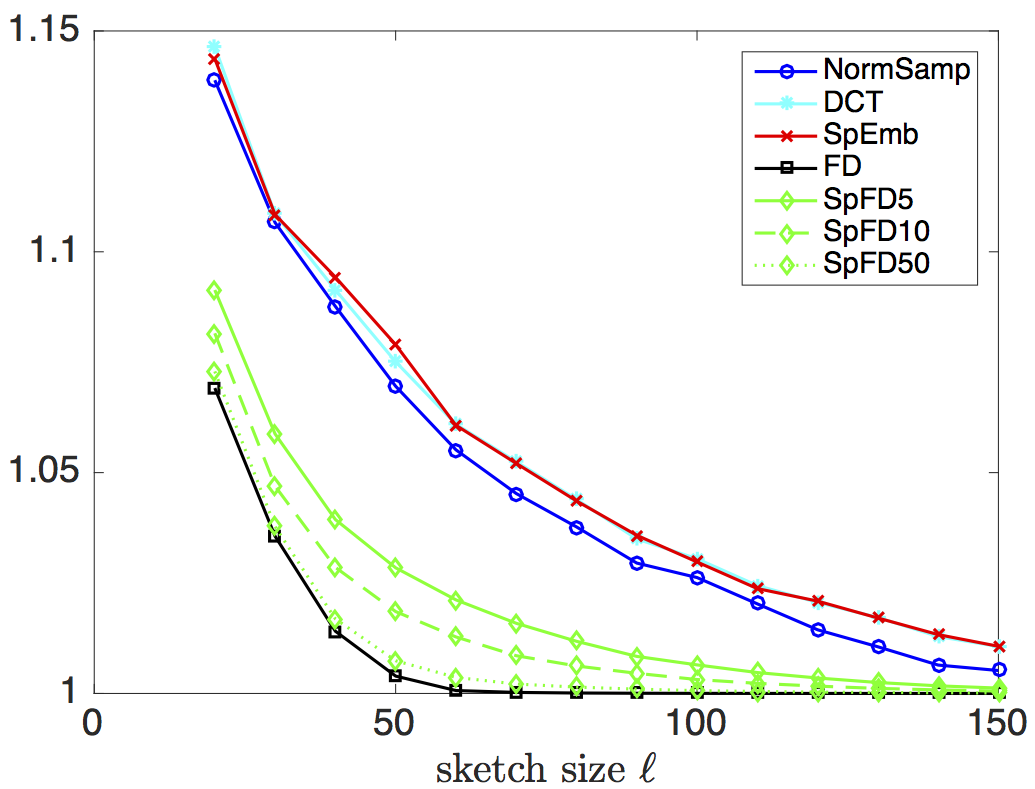}} \hspace{4pt}
\subfigure{\includegraphics[width=4.2cm,height=3.5cm]{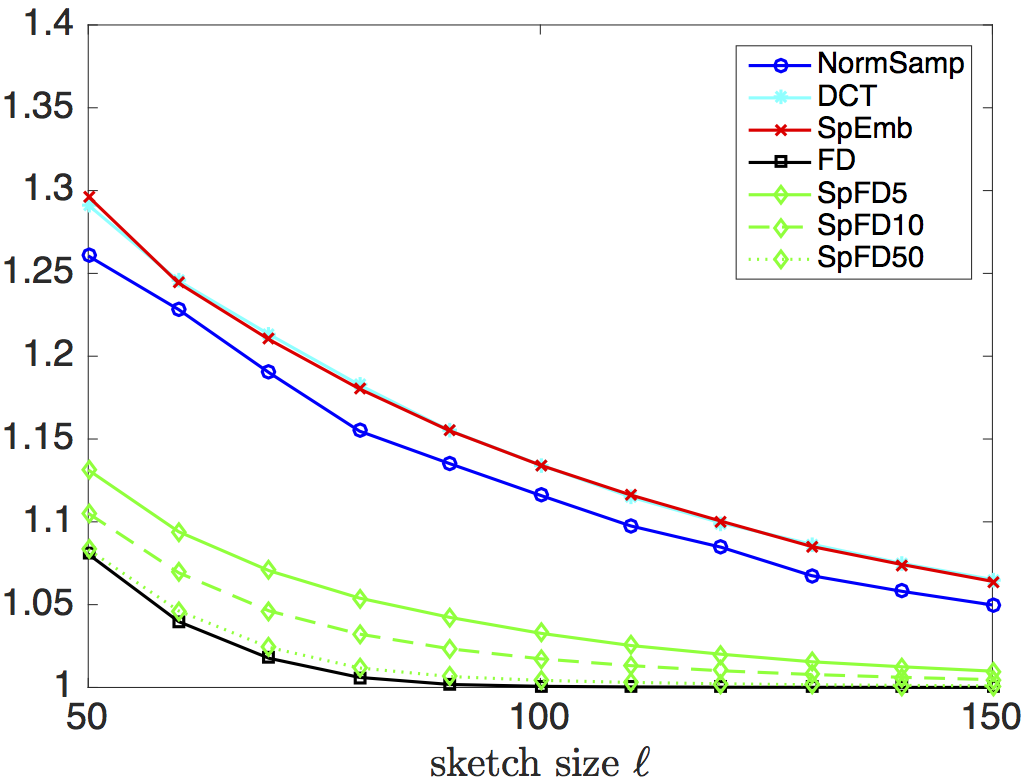}}\\
\vspace{-10pt}

\hspace{5pt}\subfigure{\makebox[15pt][r]{\makebox[20pt]{\raisebox{60pt}{\rotatebox[origin=c]{90}{$2$-norm Error}}}}%
\includegraphics[width=4.2cm,height=3.5cm]{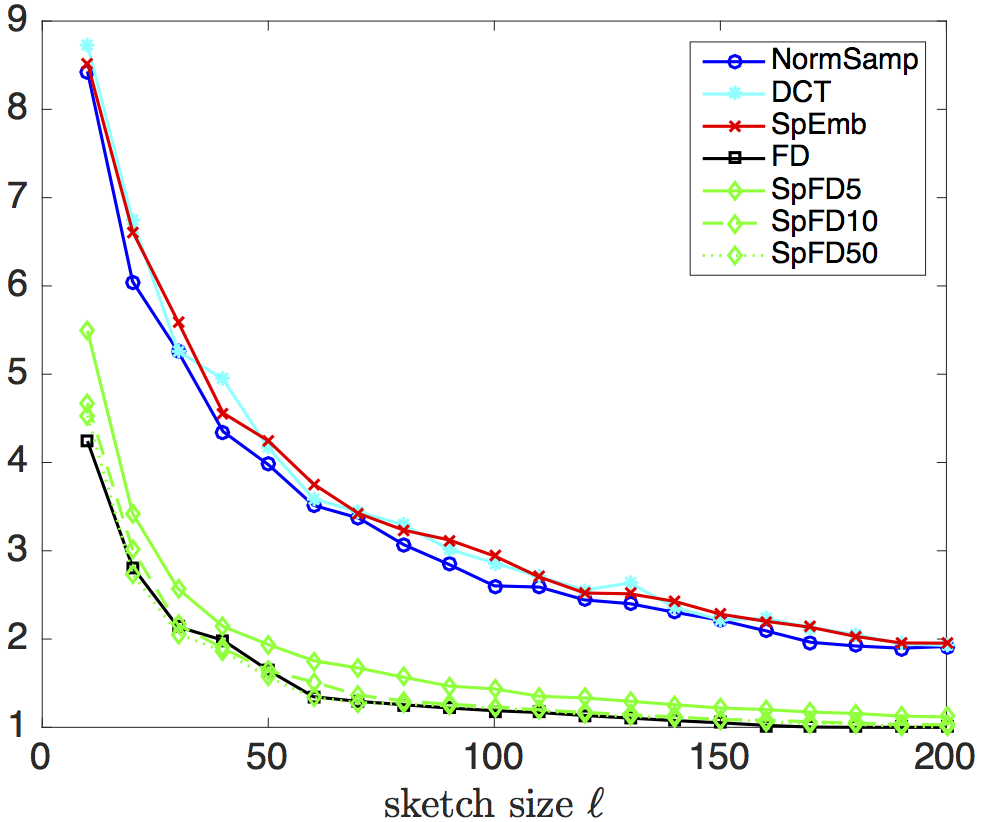}} \hspace{3pt}
\subfigure{\includegraphics[width=4.1cm,height=3.5cm]{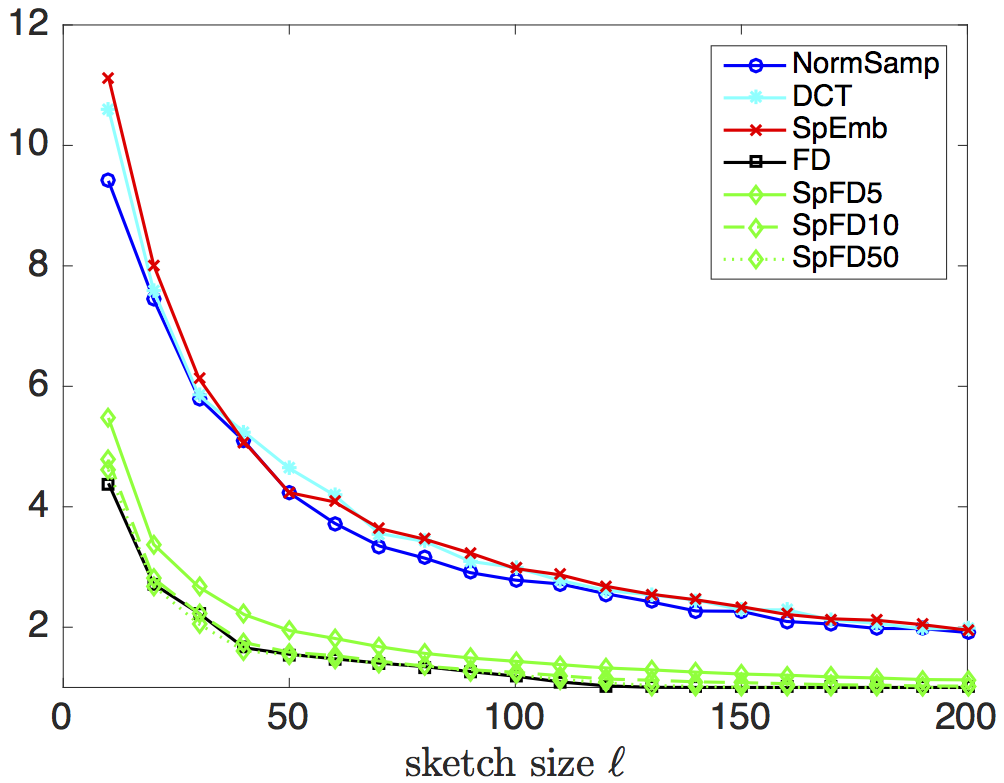}} \hspace{4pt}
\subfigure{\includegraphics[width=4.1cm,height=3.5cm]{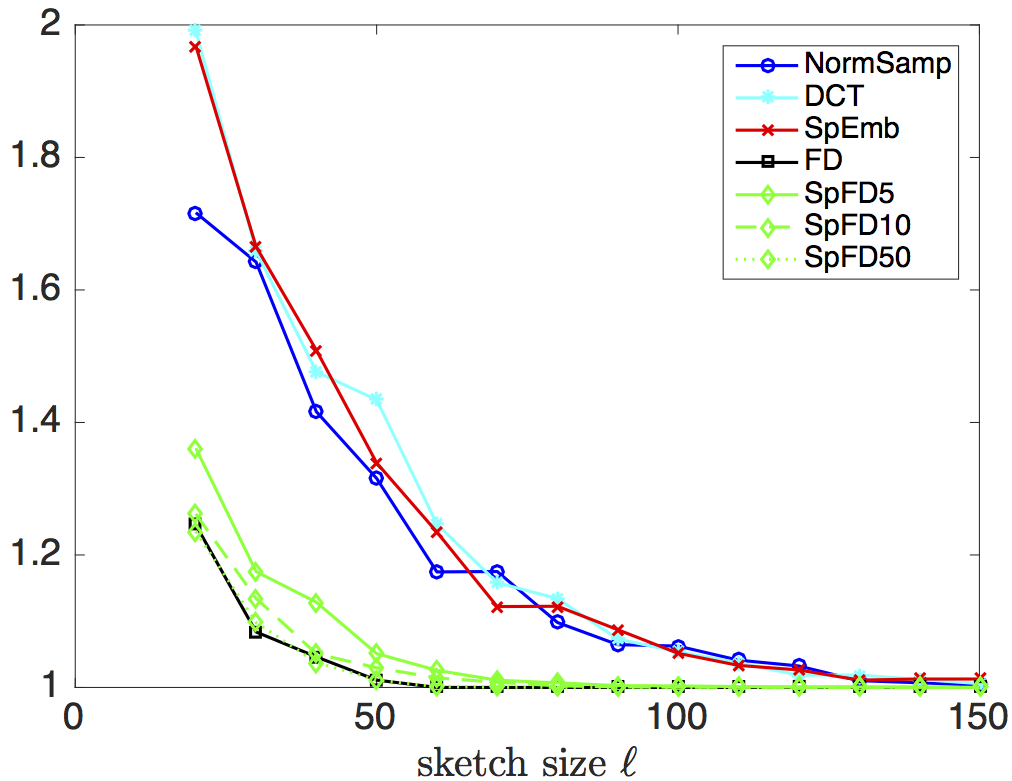}} \hspace{4pt}
\subfigure{\includegraphics[width=4.1cm,height=3.5cm]{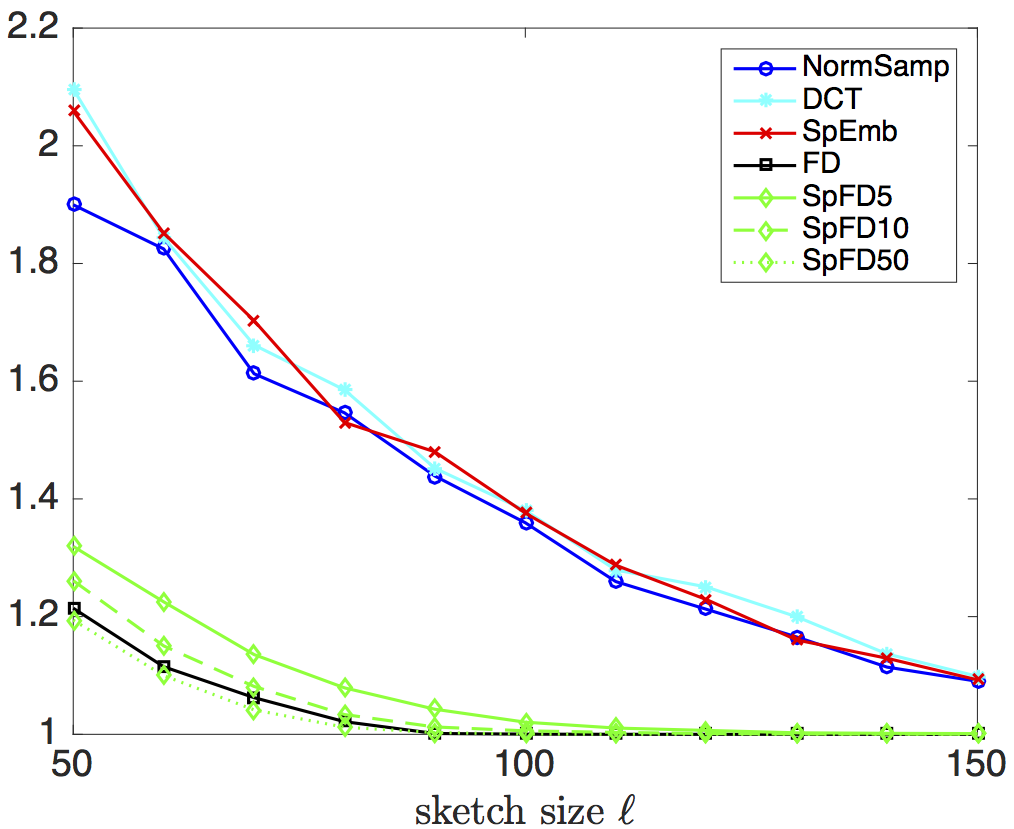}}\\
\vspace{-18pt}

\setcounter{subfigure}{4}
\subfigure[][$k=10, \zeta = 15$]{\makebox[15pt][r]{\makebox[20pt]{\raisebox{60pt}{\rotatebox[origin=c]{90}{Running Time (s)}}}}
\includegraphics[width=4.2cm,height=3.5cm]{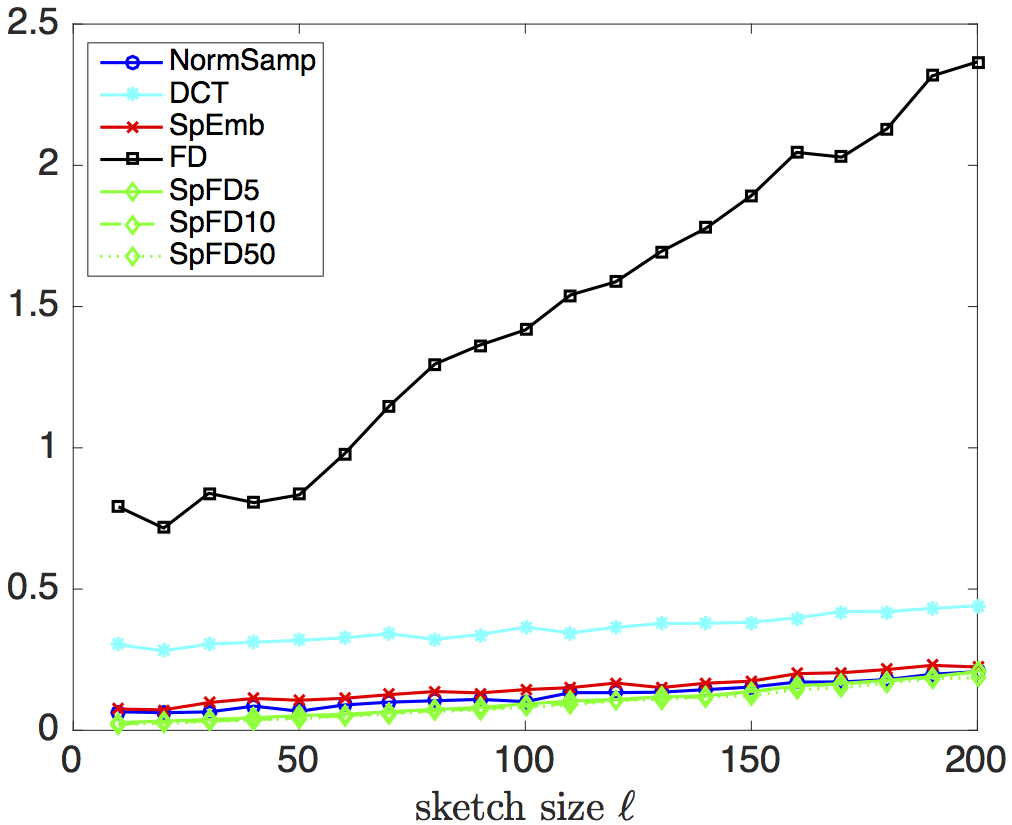}} \hspace{3pt}
\subfigure[][$k=10, \zeta = 20$]{\includegraphics[width=4.2cm,height=3.5cm]{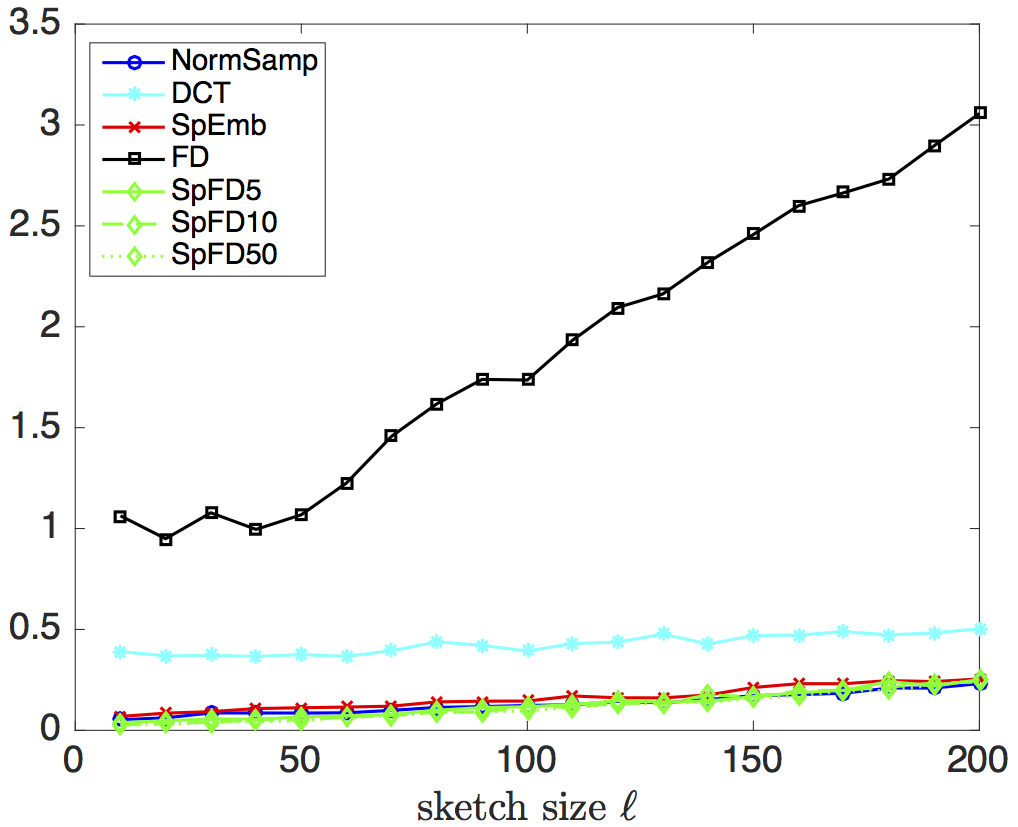}} \hspace{4pt}
\subfigure[][w8a]{\includegraphics[width=4.2cm,height=3.5cm]{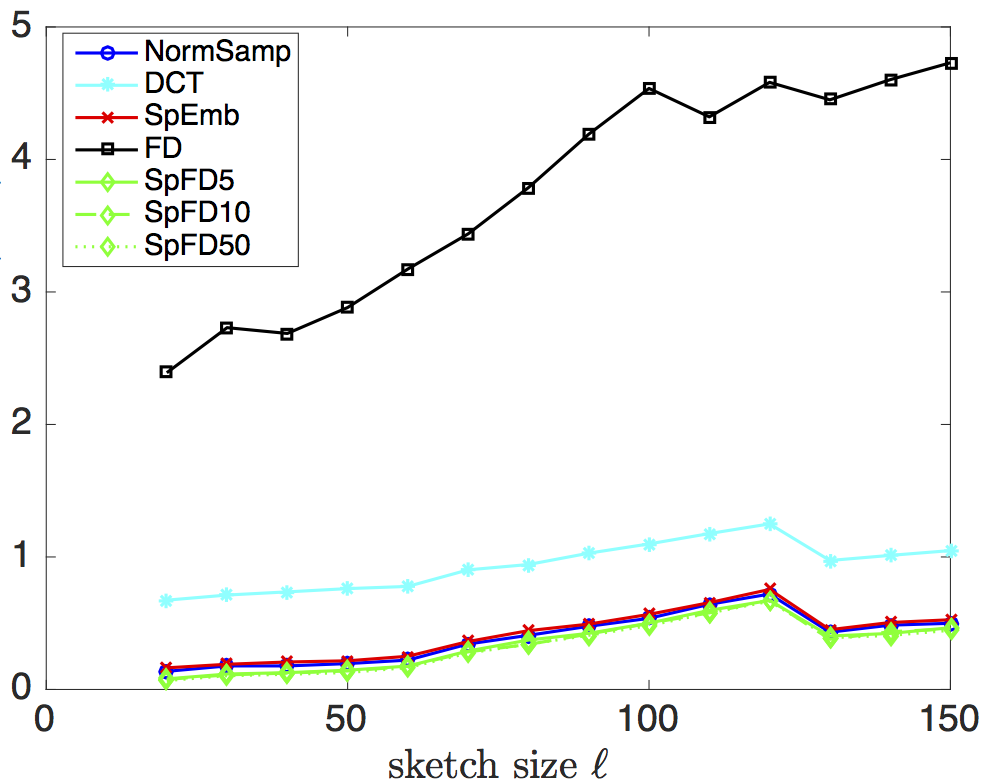}} \hspace{4pt}
\subfigure[][Birds]{\includegraphics[width=4.1cm,height=3.5cm]{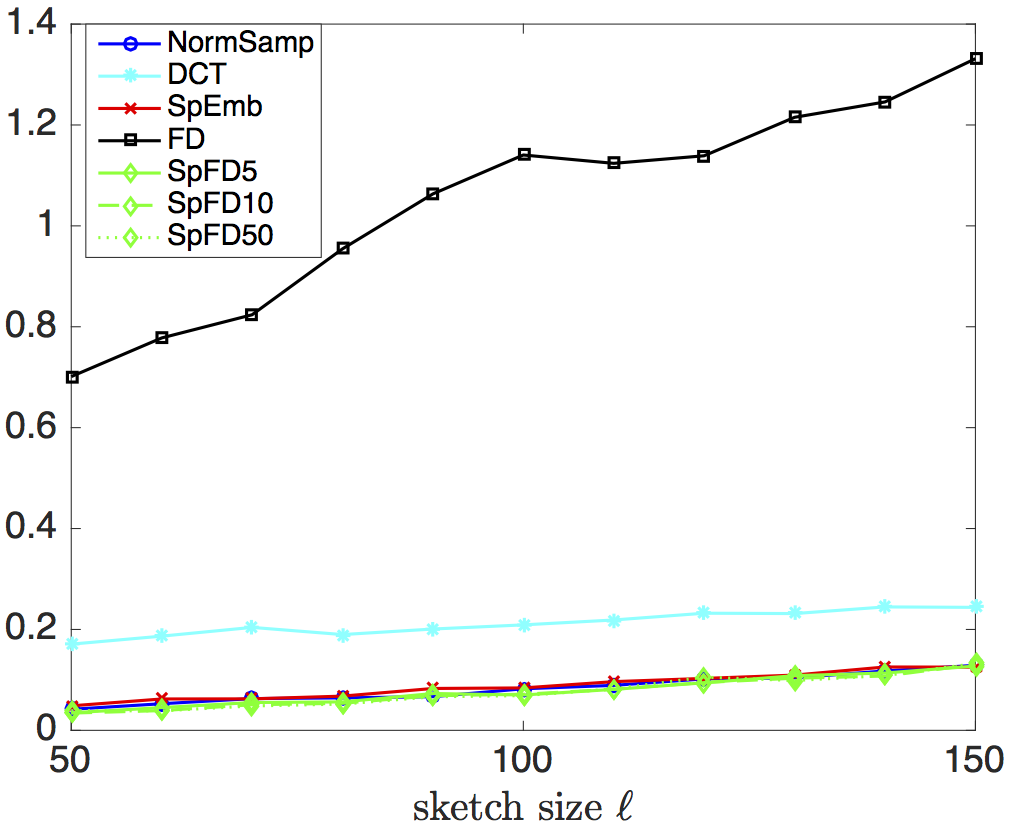}} \\

\caption{Results on synthetic datasets and real world datasets: w8a and Birds. }
\label{fig: setSynthetic}
\vspace{-10pt}
\end{figure*}

\begin{figure*}[!]
\subfigure{\makebox[15pt][r]{\makebox[20pt]{\raisebox{60pt}{\rotatebox[origin=c]{90}{$F$-norm Error}}}}
\includegraphics[width=4.2cm,height=3.5cm]{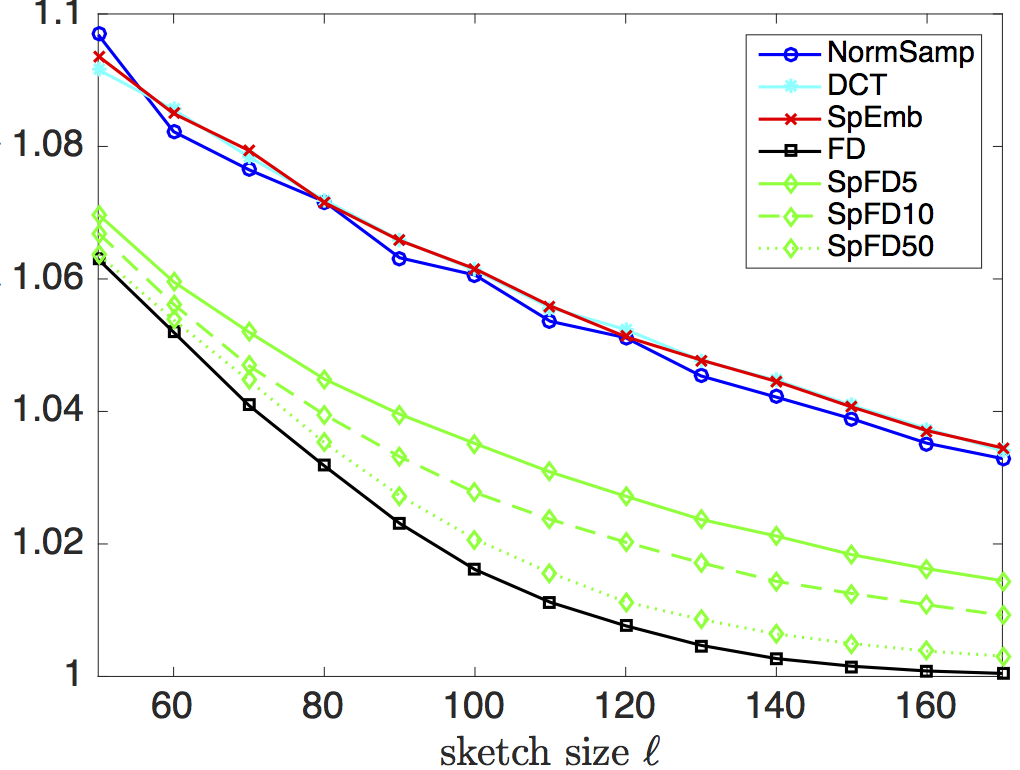}} \hspace{1pt}
\subfigure{\includegraphics[width=4.2cm,height=3.5cm]{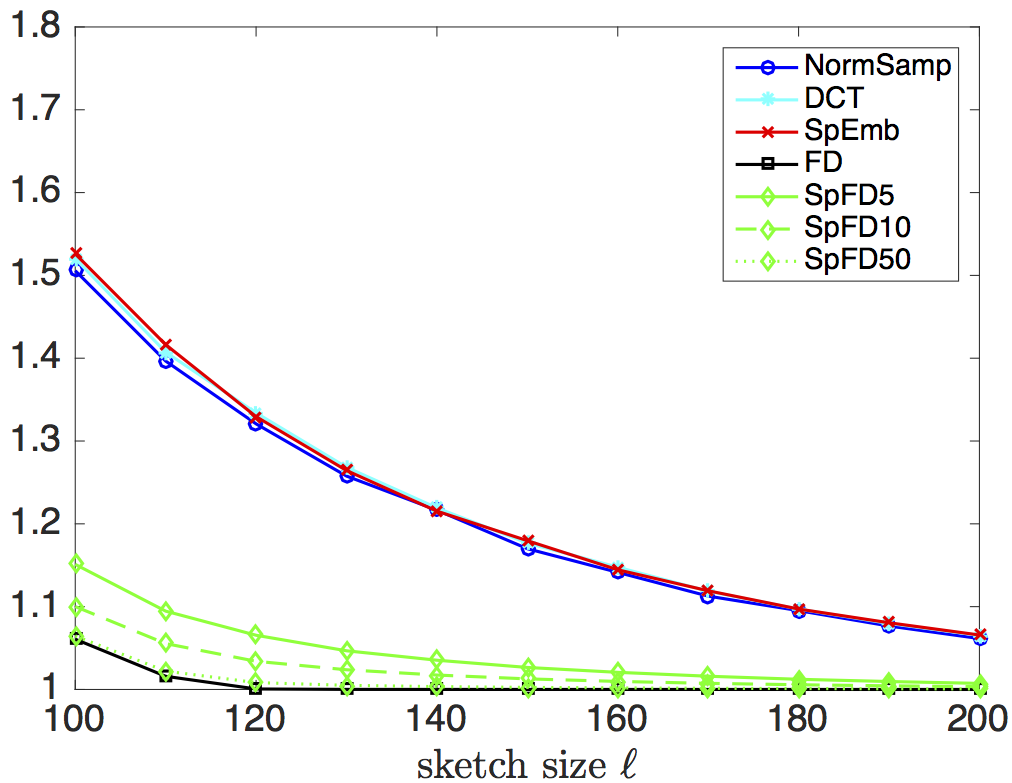}} \hspace{3pt}
\subfigure{\includegraphics[width=4.2cm,height=3.5cm]{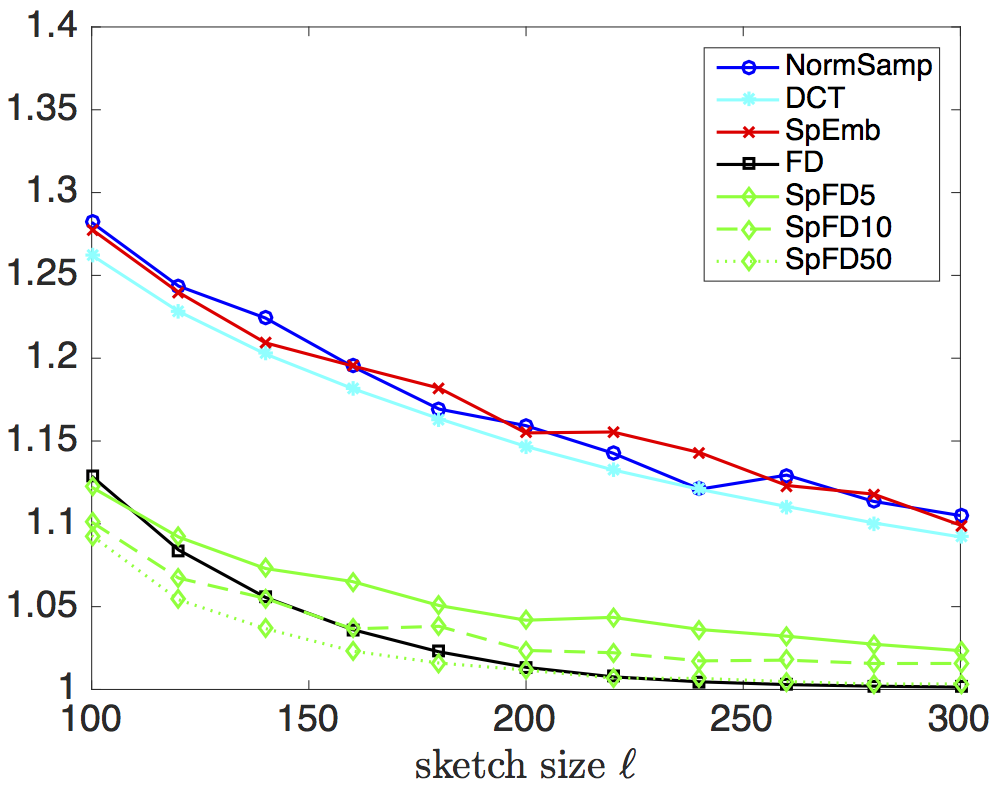}}\hspace{3pt}
\subfigure{\includegraphics[width=4.2cm,height=3.5cm]{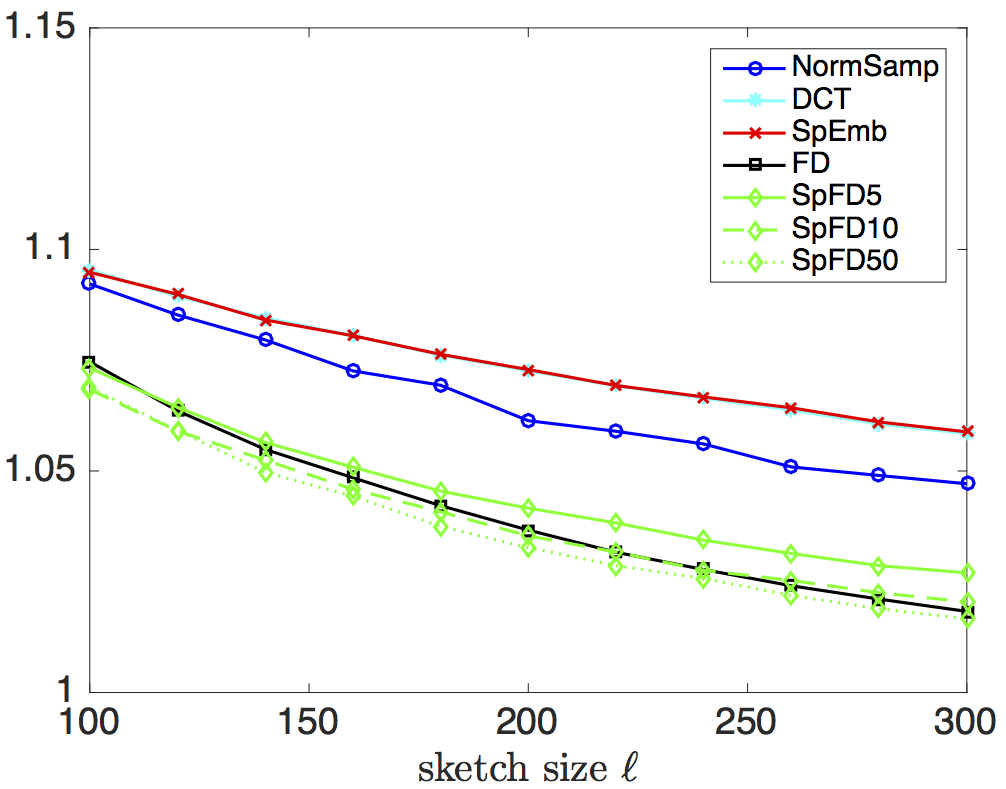}} \\
\vspace{-10pt}

\hspace{2pt}\subfigure{\makebox[15pt][r]{\makebox[20pt]{\raisebox{60pt}{\rotatebox[origin=c]{90}{$2$-norm Error}}}}
\includegraphics[width=4.15cm,height=3.5cm]{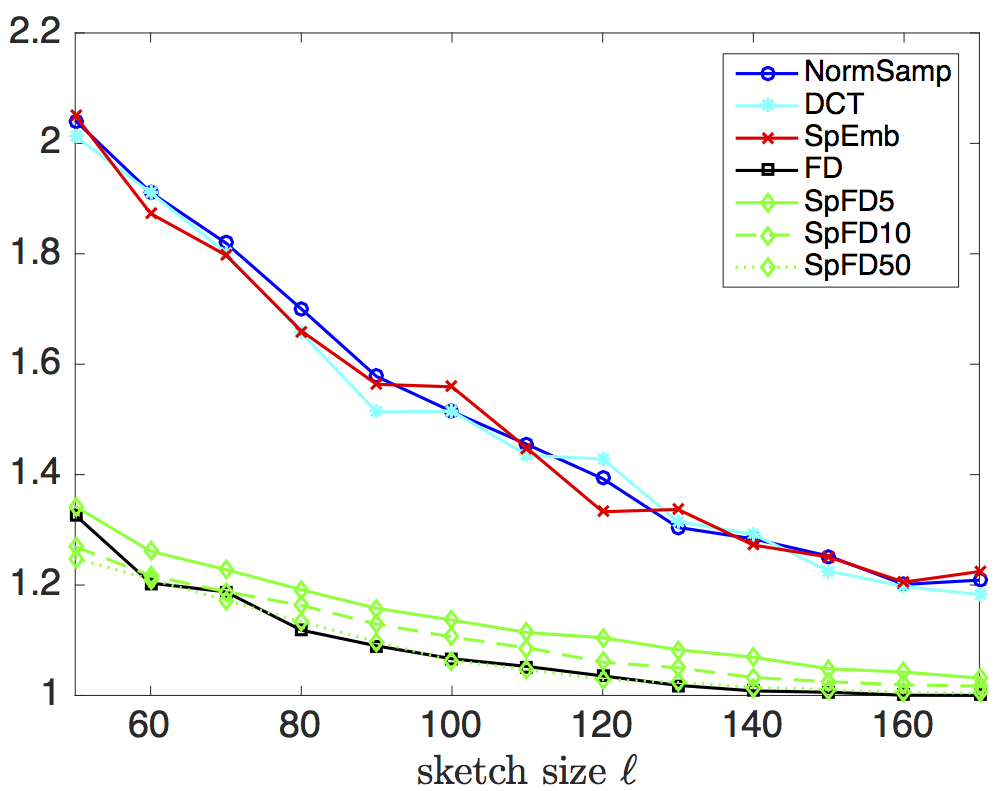}} \hspace{5pt}
\subfigure{\includegraphics[width=4.1cm,height=3.4cm]{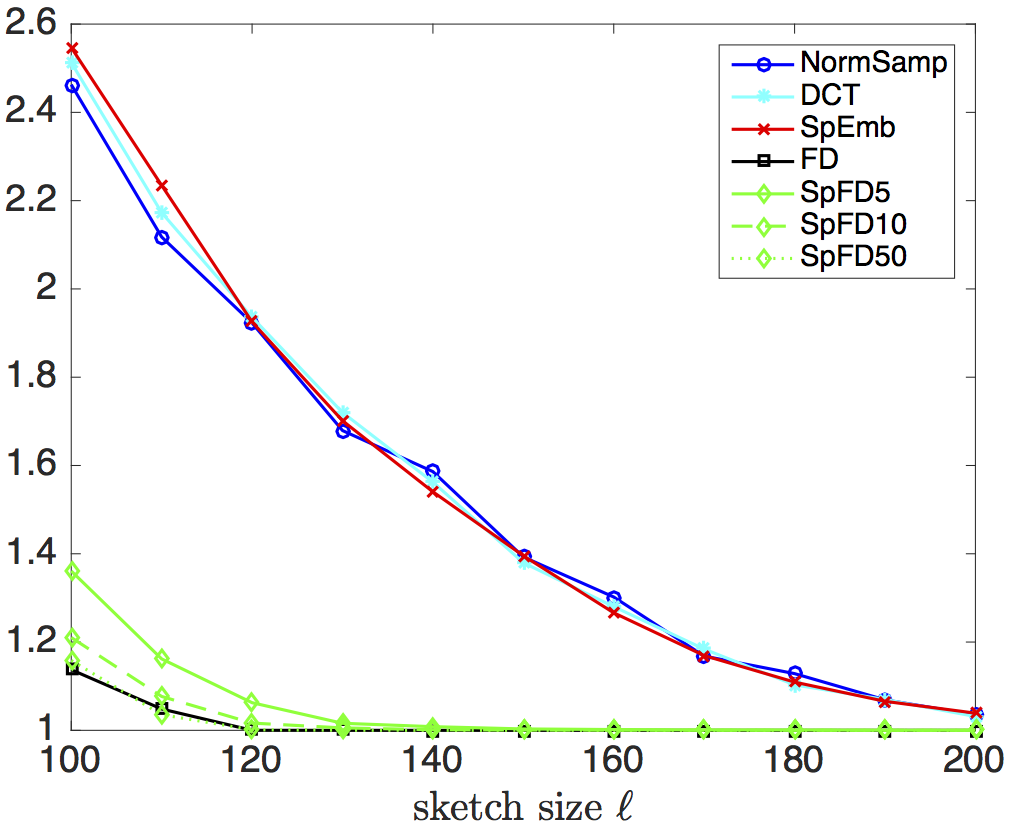}} \hspace{5pt}
\subfigure{\includegraphics[width=4.1cm,height=3.4cm]{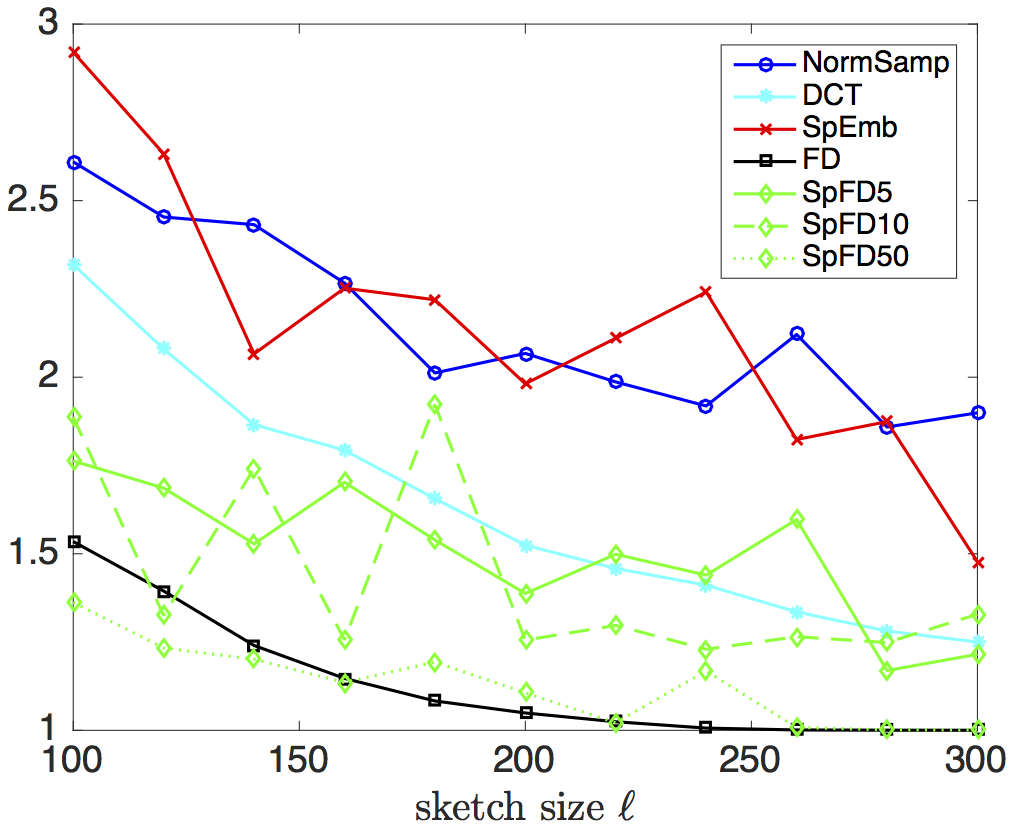}}\hspace{5pt}
\subfigure{\includegraphics[width=4.1cm,height=3.5cm]{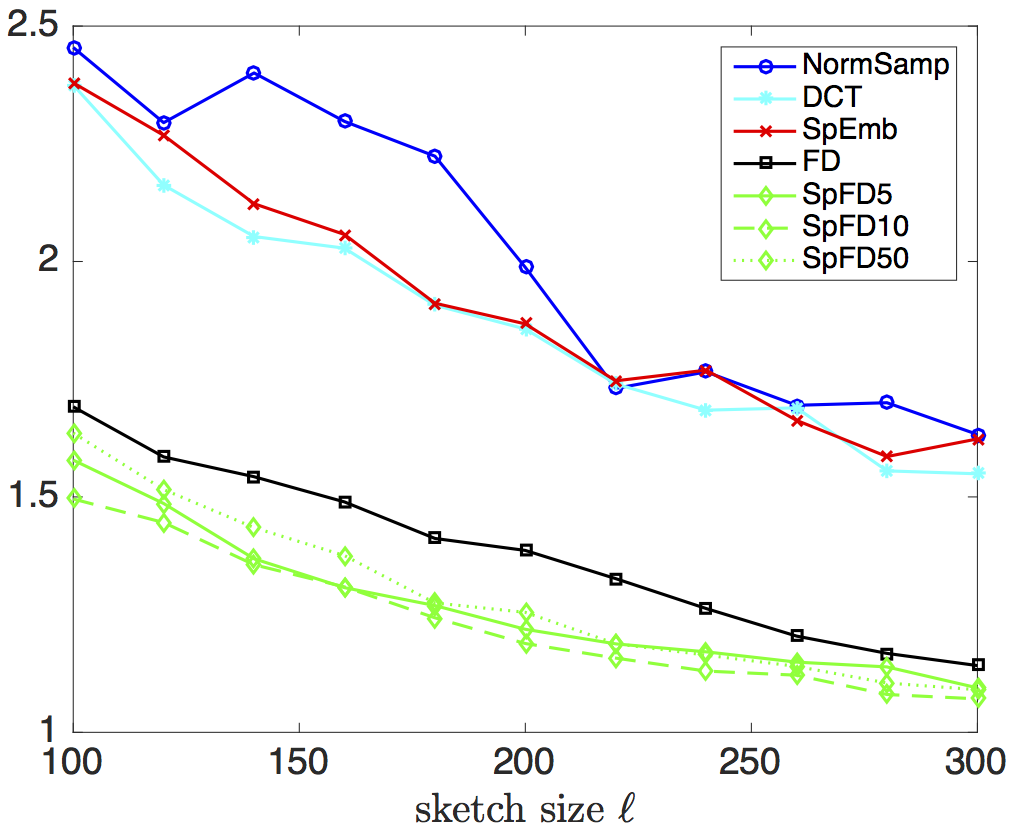}} \\
\vspace{-15pt}

\setcounter{subfigure}{8}
\subfigure[][Protein]{\makebox[15pt][r]{\makebox[20pt]{\raisebox{60pt}{\rotatebox[origin=c]{90}{Running Time (s)}}}}
\includegraphics[width=4.15cm,height=3.5cm]{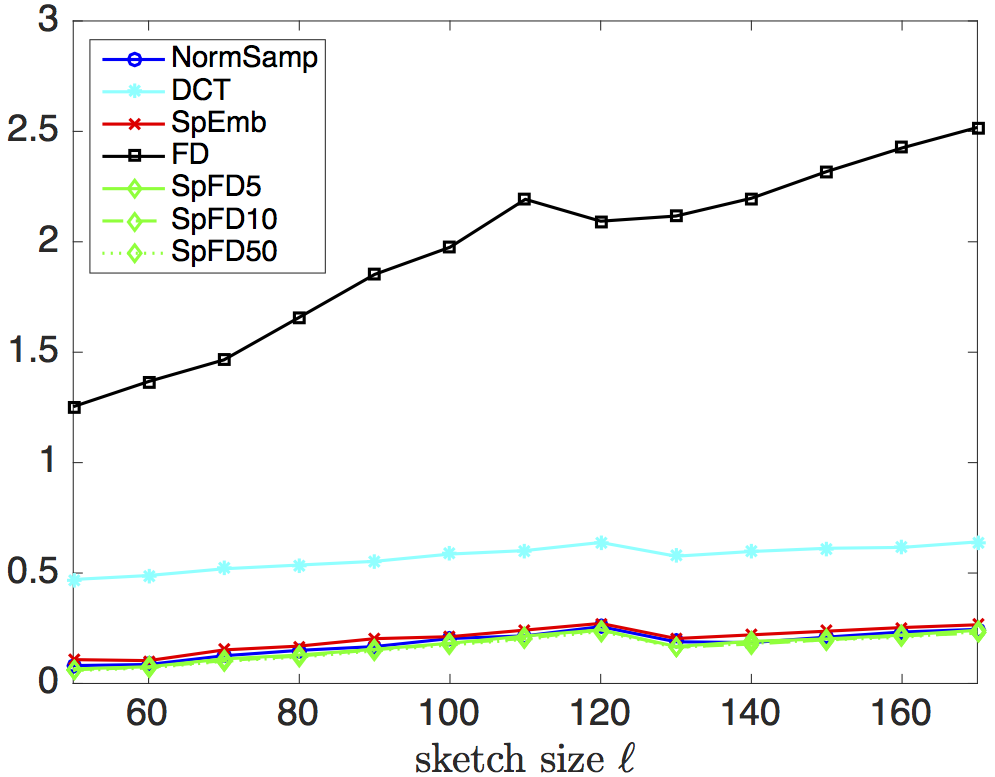}} \hspace{5pt}
\subfigure[][MNIST-all]{\includegraphics[width=4.1cm,height=3.5cm]{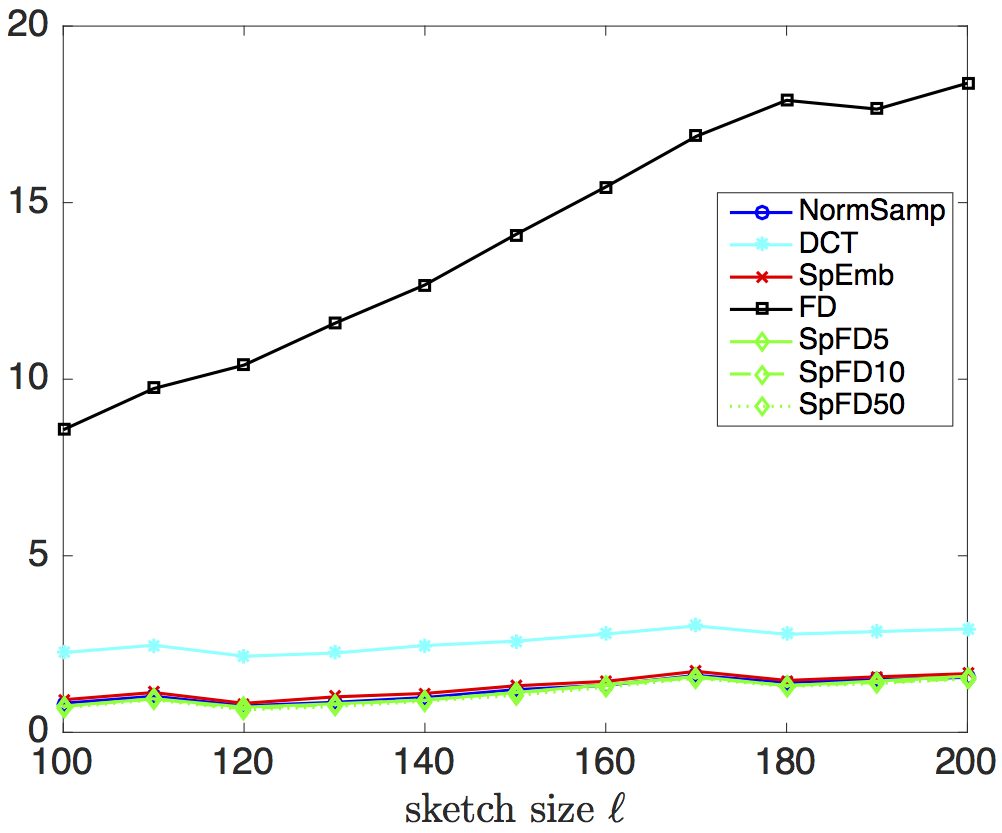}} \hspace{5pt}
\subfigure[][amazon7-small]{\includegraphics[width=4.1cm,height=3.5cm]{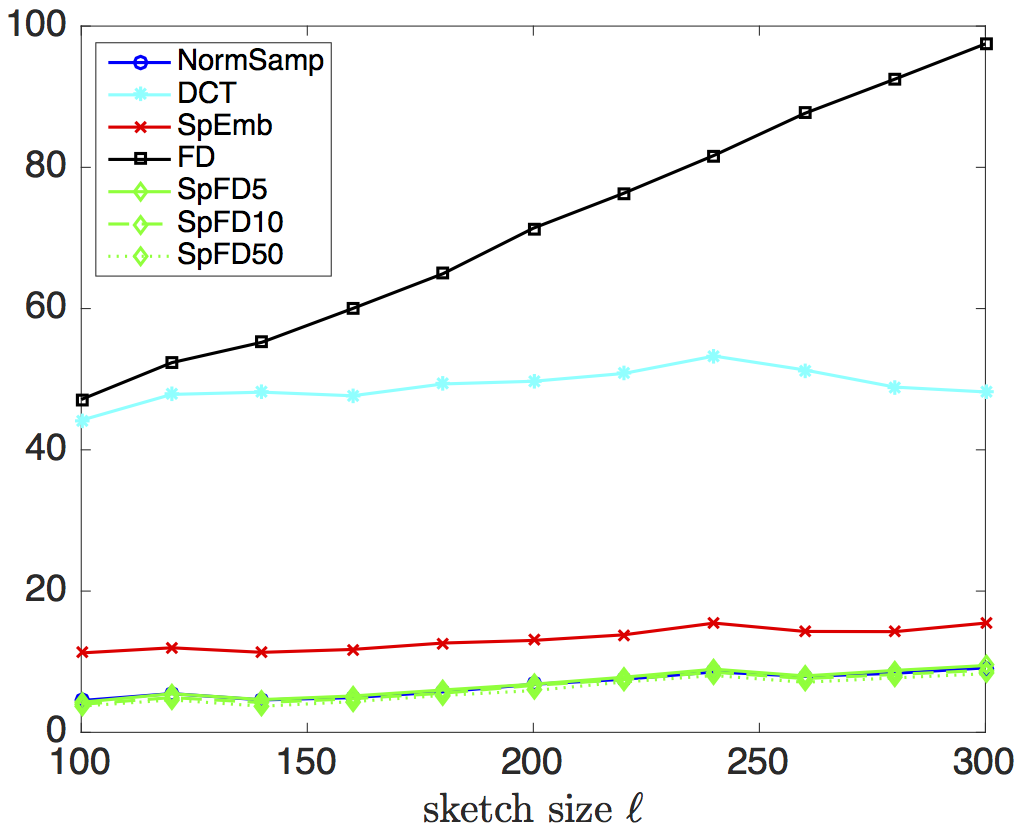}} \hspace{5pt}
\subfigure[][rcv1-small]{\includegraphics[width=4.1cm,height=3.5cm]{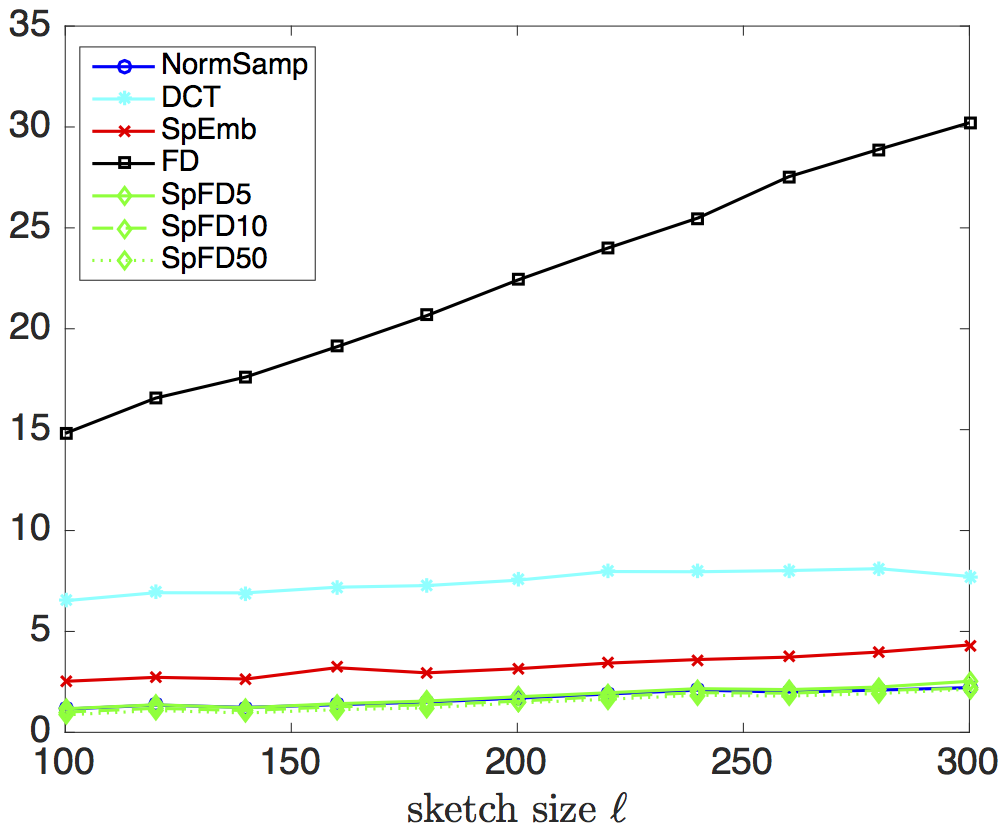}} \\

\caption{Results on real world datasets: Protein, MNIST-all, amazon7-small and rcv1-small.}
\label{fig: setReal}
\vspace{-10pt}
\end{figure*}

Regarding on the accuracy of the competing algorithms for both $F$-norm and $2$-norm, the relative errors for all algorithms decline as the sketch size $\ell$ increases. Among them, NormSamp, DCT and SpEmb achieve similar relative errors; our algorithm SpFD5, SpFD10 and SpFD50 obtain much more accurate results than the other randomized algorithms; and in general FD is the most accurate.
The accuracy level of SpFD5, SpFD10 and SpFD50 are close to that of FD.

Among SpFD5, SpFD10 and SpFD50,  the ordering of their accuracy level cannot be determined precisely.
SpFD5 contains larger error from the SpEmb it incorporates due to the smaller intermediate sketch size but less error from FD as the  number of SVD computations is reduced.
SpFD50 on the other hand, with larger error from FD and less error from the intermediate sketching.
The accuracy level of each SpFDq depends on which error dominates. Nevertheless, since FD is considered more accurate than SpEmb, we can find from our results that the accuracy level of SpFDq $(q=5, 10, 50$) improves as $q$ increases for most datasets.

For some datasets, SpFDq achieve higher accuracy level than FD.
This suggests it is possible that integrating certain level of SpEmb into FD generates a better result than the original FD.

From the running time perspective, all algorithms' running time shows a trend of inclining as the sketch size $\ell$ increases.
Among the algorithms, FD is the slowest; then followed by DCT;
NormSamp, SpEmb and SpFDq are gradually increasing with the sketch size and they are far more efficient compared to FD.

A close-up view of the running time for SpFDq ($q=5, 10, 50$), NormSamp and SpEmb on all datasets is displayed in Fig.~1 in
Section 3 of the supplemental material. NormSamp computes the probability distribution by calculating the norm square of each row that respects the sparsity of $A$. It is generally faster than projection methods like SpEmb since it does not require any row combinations to form the sketch. 
SpFDq ($q=5, 10, 50$) has similar running time as NormSamp and faster than SpEmb; in addition, the running time of SpFDq ($q=5, 10, 50$) reduces as $q$ increases for most datasets although the difference among them is not obvious. But flop counts in TABLE~\ref{tab: flopcounts} show the opposite. To understand this, possible reasons are:

\begin{itemize}
\item Demonstrated by the two tables in Section 4 of the supplemental material, the running time for computing
$[AV]_k$ and $\tilde{A}_k = [AV]_kV^T$ are similar for SpEmb and SpFDq as predicted, however the running time for constructing $V$ for SpEmb, SpFD5, SpFD10 and SpFD50 declines according to this order.
 Note that this running time consists of two parts -- forming $SA$ or $SPA$ and performing $1$ round of QR or $q-1$ rounds of SVD. One possible reason is the inclining number of zero rows in $S$, for which the reader may refer to Section 2 of the supplemental material.
 This will cut down the cost of the $q-1$ rounds of SVD for SpFDq with large $q$. The other possible reason is that as a SpEmb (-like) matrix ($S\in\mathbb{R}^{\ell\times n}$ for SpEmb, $S\in\mathbb{R}^{q\ell\times n}$ for SpFDq), the smaller the sketch size, the larger number of row combinations is required while forming $SA$ or $SPA$.

 \item Note that the earlier versions of MATLAB counted the flops. With the incorporation of LAPACK in MATLAB 6 (and newer versions),
this is no longer practical since many optimization techniques are adopted in LAPACK, and furthermore the memory latency to fetch
anything not in cache is much greater than the cost of a flop, so the flop counts in TABLE~\ref{tab: flopcounts} for each algorithm
and the running time in Fig.~1 of Section 3 in the supplemental material may not be consistent.
\end{itemize}

Based on the observation and discussions above, as the efficiency is not compromised much for large $q$, then from the accuracy perspective, we may choose a relatively large $q$ for SpFD in practice.
In conclusion, our new algorithm SpFD can be seen as an integration of FD and SpEmb in favor of FD's accuracy and SpEmb's efficiency, as a result, it is the best among the algorithms we compared in terms of accuracy and efficiency.

\section{Applications in Network Analysis}
\label{sec: 5}
Large-scale networks arise in many applications, one basic question often asked in network analysis is to identify the most `important' nodes in a network \cite{Brandes2005, Estrada2010}.
In a directed network, two types of nodes are of interest: \textit{hubs} and \textit{authorities} \cite{Kleinberg1999, Benzi2013}. Hubs are nodes which point to many nodes that are considered important while authorities are these important nodes. Thus, important hubs are those which point to many important authorities and important authorities are those pointed to by many important hubs. To measure the importance, each node has a `hub score' and an `authority score'. It is necessary to have algorithms computing the hub and authority scores associated with each node and thus determine the important nodes.

Given a directed graph $G$, $G = (V, E)$ is formed by a set of vertices $V$ and a set of edges $E$ which is formed by ordered pairs of vertices. $(i, j)\in E$ implies there is a link from node $i$ to $j$. The adjacency matrix of $G$ is a matrix $A\in\mathbb{R}^{\abs{V}\times \abs{V}}$ defined in the following way:
\[
A = (a_{ij}) \text{~ and ~~} a_{ij} = \left\{\begin{alignedat}{2}
			&1& ~ ~~~ &\text{if $(i, j)\in E$},\\
			&0& ~ ~~~ &\text{otherwise}.
			\end{alignedat}\right.
\]

Method I (HITS) \cite{Kleinberg1999, Benzi2013}: A well-known algorithm in ranking the hubs and authorities is Hypertext Induced Topics Search (HITS) algorithm.
HITS is in fact an iterative power method to compute the dominant eigenvectors for $A^TA$ and $AA^T$ which correspond to the authority and hub scores respectively.
It depends on the choice of the initial vectors
 and only the information obtained from the dominant eigenvectors of $A^TA$ or $AA^T$  is used.
In our experiments, the starting vectors we use are with all entries i.i.d. from $\mathcal{N}(0,1)$
and the stopping criterion is $\norm{x^{(k)}-x^{(k-1)}}_2\leq 10^{-3}$ where $x^{(k)}$ represents the resulting vector at the $k$th iteration.

Method II ($e^\mathcal{A}$)  \cite{Estrada2005, Benzi2013}:
The hub and authority scores of node $i$ are given by $(e^\mathcal{A})_{ii}$ and $(e^\mathcal{A})_{n+i,n+i}$
respectively for $i \in[n]$, where $
\mathcal{A} = \begin{bmatrix}0 & A \\ A^T & 0\end{bmatrix}\in\mathbb{R}^{2n\times 2n}
$. Compared to HITS, this method takes spectral information from all eigenvectors of $A^TA$ and $AA^T$ (with declining weights) into consideration which may lead to improved results.

\begin{table*}[tb]
\centering
\footnotesize
\begin{tabular}{|l|c|c|c|c|c|c|c|c|c|}
\hline
 &  HITS &
 $e^\mathcal{A}$ &
   NormSamp &
   DCT &
 SpEmb&
   FD&
  SpFD5 &
  SpFD10 &
   SpFD50 \\
 \hline
 \hline
 &
    57  &  57  &       57   &   57  &   57 &   57 &   57   &  57 &  57
 \\
\cline{2-10}
&    634   & 17 &   17 &   17  &  17  & 644  & 644  & 634  & 644
\\
\cline{2-10}
&     644 &  644   &634 &   21 &   21&   643 &   17   &644  & 119
\\
\cline{2-10}
&     721  & 643  & 644   &644  & 529&   634  & 643  & 643 &  643
\\
\cline{2-10}
 \multirow{2}{*}{Top Ten Hubs} & 643 &  634  & 643 &  643 &   54  & 106 &  634   & 17 &  106
\\
\cline{2-10}
 \multirow{2}{*}{(Computational Complexity)}  & 554 &  106  & 106 &   54  & 106 &  529 &  106  & 106 &  634
\\
\cline{2-10}
&  632 &  119  & 119  & 255&   119 &  119  &  21  & 119  &  17
\\
\cline{2-10}
 &  801  & 529  & 721 &  106&   634&    86  & 162&   255 &   86
\\
\cline{2-10}
 &  640   & 86  &  51   &119   & 51  & 162  & 529 &   86&   162
\\
\cline{2-10}
&  629   &162&    86  & 529  & 668  & 520   & 86 &  529  &  21
\\
\cline{2-10}
&\multicolumn{2}{l|}{\scriptsize no. nodes shared w. $e^{\mathcal{A}}$} & 8    & 7 &    6   &  9 &    9 &    9  &   9\\
\hline
& 719   &  1 &  717&     1   &  1   &  1    & 1   &  1   &  1
 \\
\cline{2-10}
&  717  & 315   &  2   &  2  & 673  & 315  & 315&   717 &  315
\\
\cline{2-10}
&  727  & 673  & 716&    45 &    2   &673  & 719  & 315&   148
\\
\cline{2-10}
 &  723   &148   &718   & 50 &   49  & 148 &  673&   673  &   2
\\
\cline{2-10}
\multirow{2}{*}{Top Ten Authorities} &  808  & 719  & 719 &  315  & 341  & 719 &  148  & 148 &  717
\\
\cline{2-10}
\multirow{2}{*}{(Computational Complexity)} & 735  & 717&   720 &  433  &  45 &  717  & 717 &    2 &  719
\\
\cline{2-10}
& 737    & 2  & 721 &  341&   729 &   45  & 363&   719  & 673
\\
\cline{2-10}
  &   1 &   45   &722  & 742 &  721   & 98  & 727 &  374  & 727
\\
\cline{2-10}
  & 722  & 727  & 723  & 673 &  664&   727  & 737   &808&   723
\\
\cline{2-10}
  &  770 &  534 &  724  &  98  & 148 &    2 &  534 &  727 &  735
\\
\cline{2-10}
&\multicolumn{2}{l|}{\scriptsize no. nodes shared w. $e^{\mathcal{A}}$} &3   &  5  &   5   &  9   &  8 &    8 &    8\\
\hline
 Running Time (s) & 0.0388  &  0.5248  &  0.0090    &0.0300 &   0.0100  &  0.0814   & 0.0079  &  0.0069 & 0.0062
 \\
\hline\hline
& 210  &  210  & 210 &  210 &210&     210&   210 & 210 &  210
 \\
\cline{2-10}
&            637      &   637       &  637     &    637   &      637    &     637      &   637     &    637     &    637
\\
\cline{2-10}
 &  413       &  413    &     413     &    413     &    413     &    413     &    413 &        413   &      413
\\
\cline{2-10}
& 1586     &   1586      &  1586       & 1586       &  552    &    1586     &   1586 &       1586   &     1586
\\
\cline{2-10}
\multirow{2}{*}{Top Ten Hubs}&  552    &     552         &552     & 542  &   1586   &   552   &   552     &552     &    552
\\
\cline{2-10}
\multirow{2}{*}{(Death Penalty)}& 462        & 462      &   930       &  930   &      545      &   930  &       462  &       462    &     618
\\
\cline{2-10}
&930     &    930     &    542       & 1275      &   618     &    542      &   618     &    542     &    930
\\
\cline{2-10}
&  542     &    542       & 1275      &   545       &  462      &   462   &      542  &       618     &    542
\\
\cline{2-10}
& 618     &    618       &  462   &       80     &    578      &   618  &       930 &        930  &       462
\\
\cline{2-10}
& 1275       & 1275      &    80     &   1258     &   1275    &    1275   &     1275   &     1275    &    1275
\\
\cline{2-10}
& \multicolumn{2}{l|}{\scriptsize no. nodes shared w. $e^{\mathcal{A}}$} &  9&7 &   8   & 10  &  10 & 10 &   10 \\
\hline
 &4       &    4       &   16      &     4     &      1       &    4     &      4   &        4    &       1
 \\
\cline{2-10}
 &1      &     1       &    4     &      6      &     4    &       1  &         1  &         1      &     4
\\
\cline{2-10}
 &6      &     6        &  14     &      1      &     3   &        6      &     6        &   6       &    6
\\
\cline{2-10}
 & 7        &   7       &   27        &  27      &     5     &      7    &      16  &        16        &   7
\\
\cline{2-10}
\multirow{2}{*}{Top Ten Authorities} & 10    &      10        &  10      &    16     &     27     &     10       &    2         &  7   &       10
\\
\cline{2-10}
\multirow{2}{*}{(Death Penalty)} & 16       &   16     &      1      &   130        &  21       &   16      &    10  &         2      &    16
\\
\cline{2-10}
 &  2       &    2      &    13      &    12    &      14  &         3      &    14        &   3    &       3
\\
\cline{2-10}
& 3       &    3     &      9      &   513     &   1632      &     2     &     44    &      10      &     2
\\
\cline{2-10}
 & 44  &  44     &     3        &   2      &   394   &       44        &   7     &     27    &      27
\\
\cline{2-10}
& 27 &    27  &   2 &   394  &    12  &     27  &   384  &     9 &   44
\\
\cline{2-10}
& \multicolumn{2}{l|}{\scriptsize no. nodes shared w. $e^{\mathcal{A}}$} & 7  &  6   &4   &   10 &     8  &    9   &  10  \\
 \hline
 Running Time (s)&   0.0723  &  5.1974 &   0.0282 &   0.1312 &   0.0431   & 0.2868 &   0.0213  &  0.0197 &   0.0188
 \\
\hline
\end{tabular}
\caption{The top 10 hubs, top 10 authorities and running time used to derive the hub and authority scores for networks \textit{Computational Complexity} and \textit{Death Penalty}.}
\label{tab: CC}
\vspace{-5pt}
\end{table*}

Method III: Low rank approximation techniques could be applied to compute the hub and authority scores.
For a randomized algorithm with sketching matrix $S$, the orthogonal basis $V\in\mathbb{R}^{n\times \ell}$ for the row space of the sketch $SA\in\mathbb{R}^{\ell\times n}$ is obtained similar as in Algorithm~\ref{alg: proto}.
The orthogonal bases for FD and SpFD are obtained following Algorithm~\ref{alg: fd} and Algorithm~\ref{alg: SpFD} respectively.
The sketch size $\ell$ is chosen based on $k$, the number of top hubs and authorities required in a question. Then following Algorithm~\ref{alg: SVD}, we could obtain the approximated SVD of $A$, $\tilde{U}\tilde{\Sigma}\tilde{V}^T$ where $\tilde{U}, \tilde{V}\in\mathbb{R}^{n\times \ell}$ and $\tilde{\Sigma}\in\mathbb{R}^{\ell\times \ell}$.
\begin{algorithm}[!]
\caption{Approximating SVD \cite{Halko2011}}
\label{alg: SVD}
\begin{algorithmic}[1]
\REQUIRE $A\in\mathbb{R}^{n\times n}$ and $V\in\mathbb{R}^{n\times \ell}$.
\ENSURE $\tilde{U}, \tilde{V}\in\mathbb{R}^{n\times \ell}$ and $\tilde{\Sigma}\in\mathbb{R}^{\ell\times \ell}$.
\STATE Form $B = AV\in\mathbb{R}^{n\times \ell}$, where $BV^T$ yields the low rank approximation of $A$: $A\approx BV^T = AVV^T$.
\STATE Compute SVD of $B = \tilde{U}\tilde{\Sigma}\hat{V}^T$.
\STATE Set $\tilde{V} = V\hat{V}$.
\end{algorithmic}
\end{algorithm}

Note that since
$e^\mathcal{A}=~ \begin{bmatrix} U \cosh(\Sigma)U^T & U\sinh(V^T)\\
		V\sinh(\Sigma)U^T & V\cosh(\Sigma)V^T \end{bmatrix} $
where $A = U\Sigma V^T$ is the SVD of $A$.
We could use $\left(\tilde{U}\cosh(\tilde{\Sigma})\tilde{U}^T\right)_{ii}$ and $\left(\tilde{V}\cosh(\tilde{\Sigma})\tilde{V}^T\right)_{ii}$ to approximate the hub and authority scores of the $i$th node, respectively. In the experiments, we use $\ell = k + p$ where $p$ is set to be $5$.

We again compare the five algorithms NormSamp, DCT, SpEmb, FD and SpFDq ($q=5, 10, 50$) against the HITS and $e^{\mathcal{A}}$ methods.
All methods are applied to the following two real world networks.
\begin{itemize}
\item \textit{Computational Complexity}: The dataset contains $884$ nodes and $1616$ directed edges.
\item \textit{Death Penalty}: The dataset contains $1850$ nodes and $7363$ directed edges.
\end{itemize}

The rank of the decomposition used in the standard algorithms for the original hubs and authorities is $k=1$ \cite{Kleinberg1999}. In our experiments, we consider a more general case, that is, for each network we obtain the top $10$ hubs and authorities ($k=10$).
The results  are summarized in TABLE~\ref{tab: CC}.

For network \textit{Computational Complexity}, HITS and $e^\mathcal{A}$ produce different rankings, they share 4 out of the top 10 hubs and 4 out of the top 10 authorities. This discrepency is expected with reasons as discussed earlier.
As the low rank approximation methods follow the computation of $e^\mathcal{A}$, the results are compared against $e^\mathcal{A}$. For the hub ranking, SpFD5, SpFD10, SpFD50 and FD generate 9, 9, 9 and 9 out of 10 top hubs, respectively.
In terms of authority ranking,
SpFD5, SpFD10, SpFD50 and FD have 8, 8, 8, 9 out of 10 top authorities respectively projected by $e^\mathcal{A}$.

For network \textit{Death Penalty}, HITS and $e^\mathcal{A}$ generate the same top 10 hubs and 10 authorities with the same orderings. FD, SpFD5, SpFD10 and SpFD50 all have the same top 10 hubs as HITS and $e^\mathcal{A}$ with slight variations on the ordering.  Similarly for authority ranking, FD and SpFD50 identify all top 10 authorities, followed by SpFD10 and SpFD5 with 9 and 8 out of 10 top authorities identified respectively.

In comparison, SpFDq and FD are more accurate than NormSamp, DCT and SpEmb and this outcome aligns with the results obtained in Section~\ref{sec: 4}. In terms of running time, SpFDq are more efficient than DCT and FD and similar as NormSamp and SpEmb which is also in line with the results in Section~\ref{sec: 4}.  From this exercise, it is obvious that the low rank approximation could be applied to the network problem of identifying hubs and authorities. According to the accuracy level and efficiency, we have seen that SpFDq ($q = 5, 10, 50$) perform the best among the competing algorithms.

\section{Concluding Remarks}\label{sec: 6}
In this paper, we have developed a fast frequent directions algorithm SpFD for the low rank approximation problem. It makes use of the natural block
structure of FD and incorporates the idea of SpEmb in it and because of that, SpFD is able to accelerate the relatively inefficient algorithm FD
by a great amount. On the other hand, SpFD obtains a good low rank approximation with sketch size linear on the rank approximated which implies it is
more accurate than SpEmb. These two points are well supported by our experimental results and the application on network problems.

Many interesting issues are worthy of further study. These issues include:

\begin{itemize}
\item
Based on the performance of SpFD in our experiments, the proved error bound  in Theorem 1 does not seem to be tight. The techniques we used to derive SpFD's error bound
still largely depend on the properties derived for FD, it is important to seek new ways to derive a better  error bound by further exploiting the well-defined structure
of SpFD, this is still under investigation.

\item
An alternative to using FD would be to form the strong rank-revealing QR decomposition for a sparse sketch of $A$ (which could be terminated once the diagonal entries of the triangular/trapezoidal factor are small enough). Another thread would be to replace FD with other randomized algorithms such as in \cite{Halko2011}. An advantage of FD and our SpFD is that the matrix can be processed sequentially. To improve the performance, other randomized techniques \cite{Mahoney2011} or  sparse sketching methods \cite{Urano2013} could be used to form the sketch. They would have similar theoretical guarantees proved in this paper.
This is a subject for our future work. 

\item
All algorithms for SVD are necessarily iterative (the number of iterations required may be small and the flop counts
depend on the machine precision). The data needs to be clean in order for SpFD and FD to work well for big datasets. In practice, noisy big datasets require extra power iterations to attain acceptable accuracy \cite{Halko2011}. This has been demonstrated, for example,  by a method named SFD in  \cite{Ghashami2016}.
SFD in \cite{Ghashami2016} is a variation of FD which introduces a lower bound on the number of nonzero entries being processed in each iteration and it adopts power iteration
to boost the accuracy of the shrinking. Numerical results in \cite{Ghashami2016} has shown that SFD is very efficient for highly sparse matrices, so we only compare the accuracy of SFD, FD and SpFDq for three sparse datasets w8a, Birds and rcv1 in Fig.~\ref{fig: setSFD}
 (the number of power iterations in SFD is set to be $1$). Clearly SFD achieves a higher accuracy level for both $2$-norm and $F$-norm compared to FD and SpFDq. \\
It is critical to measure accuracy for big data using the spectral norm,
as discussed in the appendix of \cite{LiLSSKT2017}.  To the best of our knowledge, there is still a lack of good error bounds measured by spectral norm
for FD and many randomized methods. According to our numerical results, SpFD and FD work well for both Frobenius norm and spectral norm approximation.
We are now working on incorporating power iterations into randomized sketch to derive good spectral norm error bounds.

\begin{figure*}[!]
\centering

\subfigure{\makebox[15pt][r]{\makebox[20pt]{\raisebox{50pt}{\rotatebox[origin=c]{90}{$F$-norm Error}}}}
\includegraphics[width=4.2cm,height=3.1cm]{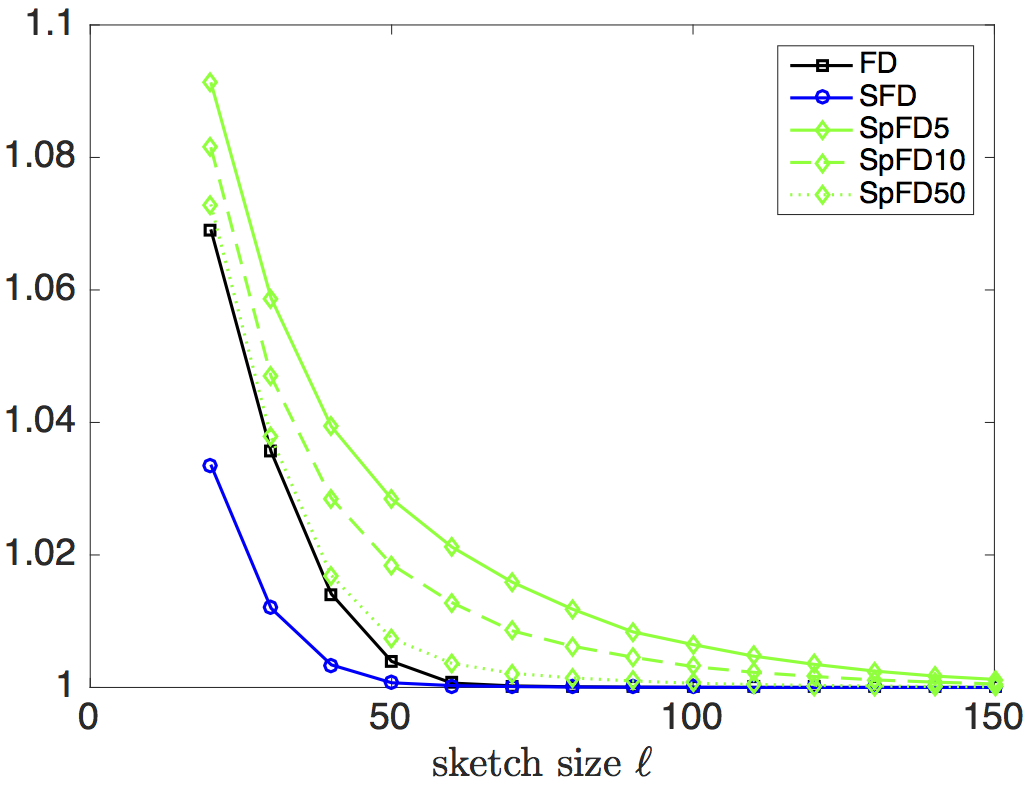}} \hspace{1pt}
\subfigure{\includegraphics[width=4.2cm,height=3.1cm]{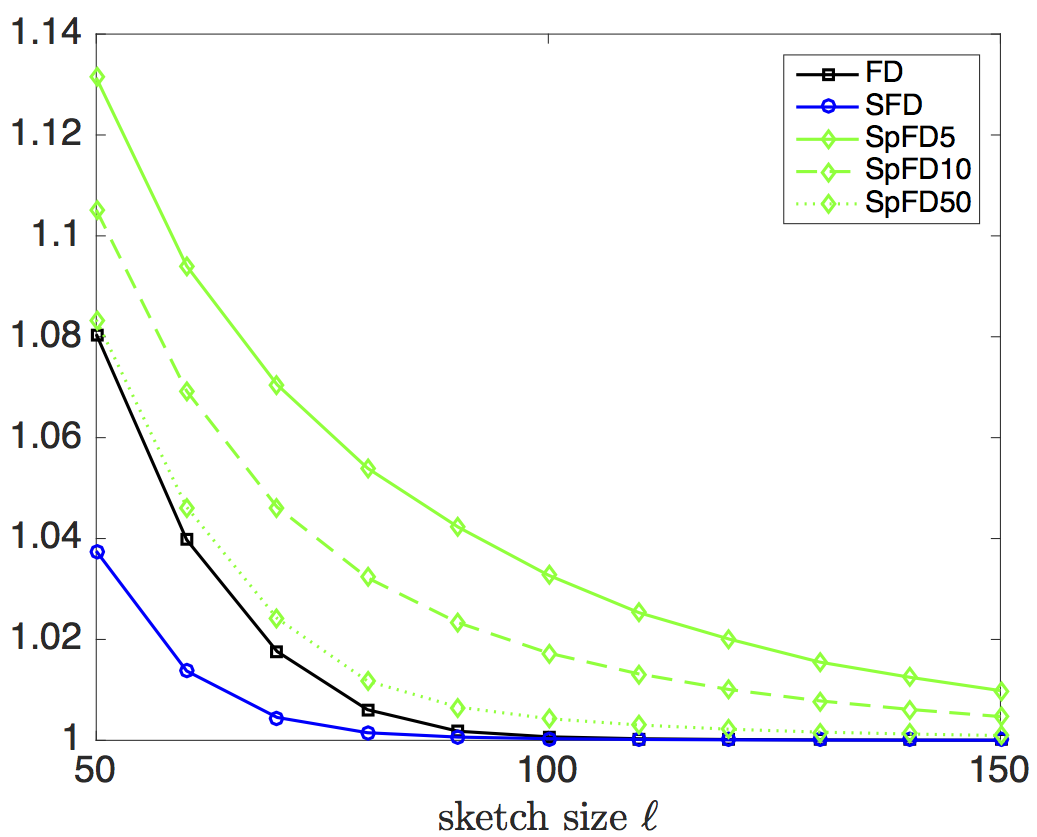}}\hspace{1pt}
\subfigure{\includegraphics[width=4.2cm,height=3.1cm]{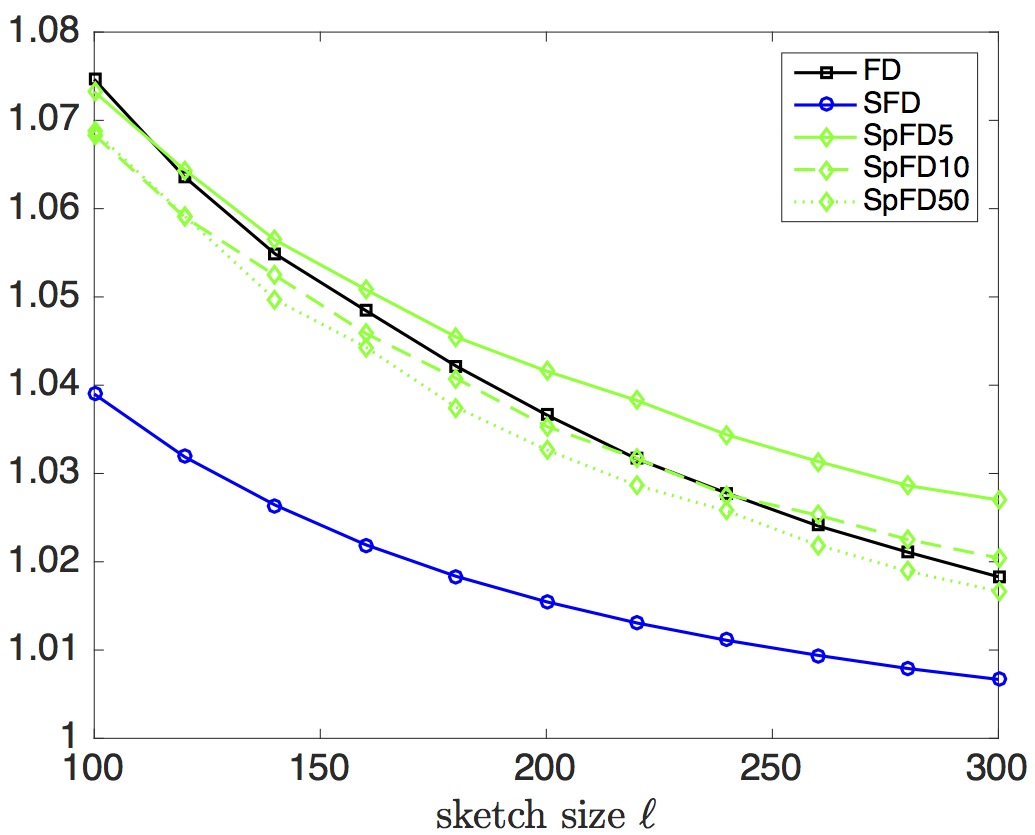}}  \\
\vspace{-5pt}
\setcounter{subfigure}{0}
\hspace{2pt}\subfigure[][w8a]{\makebox[15pt][r]{\makebox[20pt]{\raisebox{50pt}{\rotatebox[origin=c]{90}{$2$-norm Error}}}}
\includegraphics[width=4.15cm,height=3.1cm]{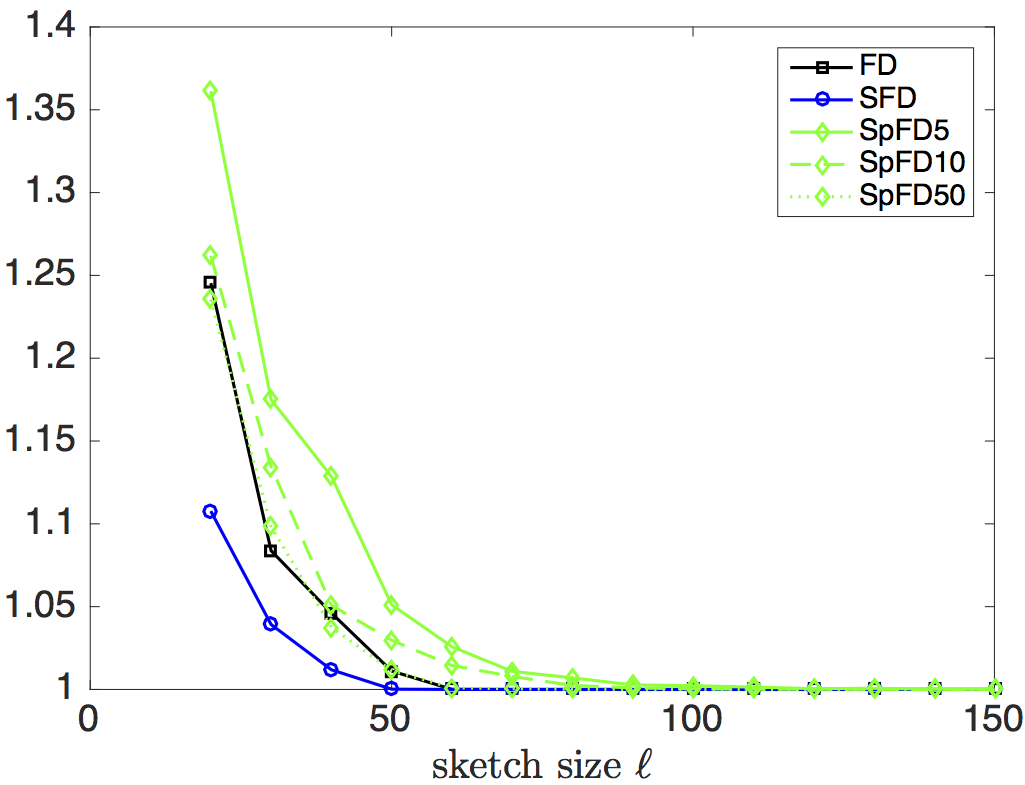}} \hspace{5pt}
\subfigure[][Birds]{\includegraphics[width=4.1cm,height=3.1cm]{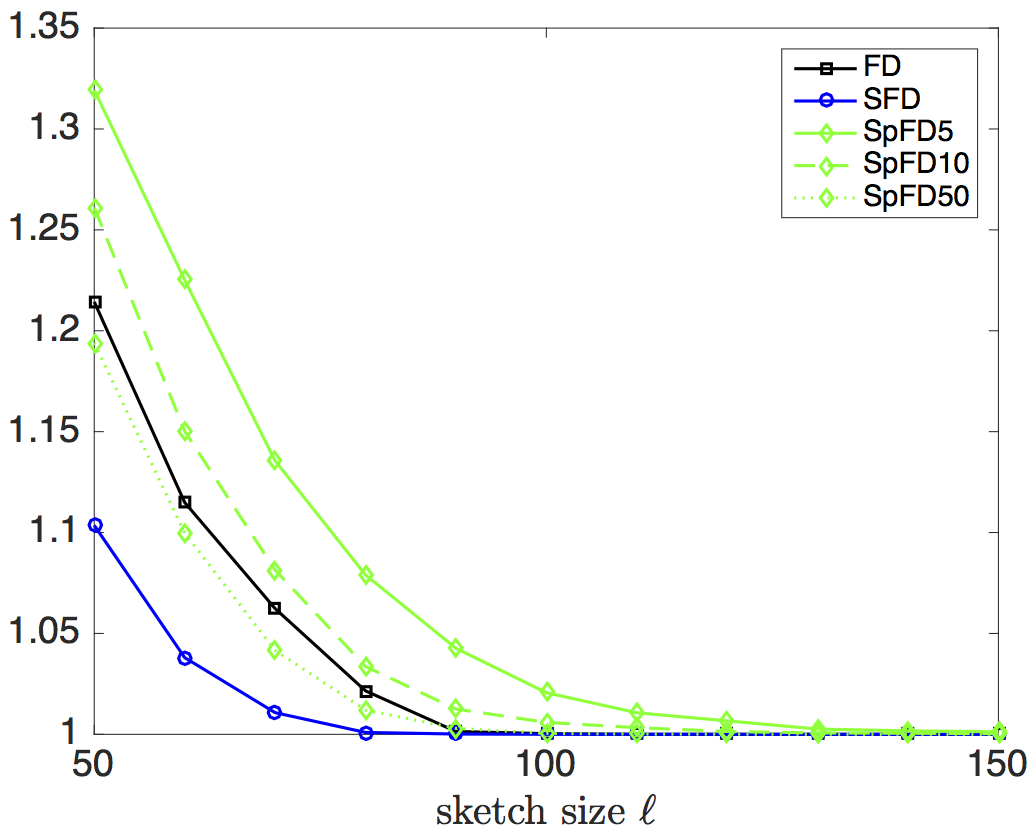}}\hspace{5pt}
\subfigure[][rcv1]{\includegraphics[width=4.1cm,height=3.1cm]{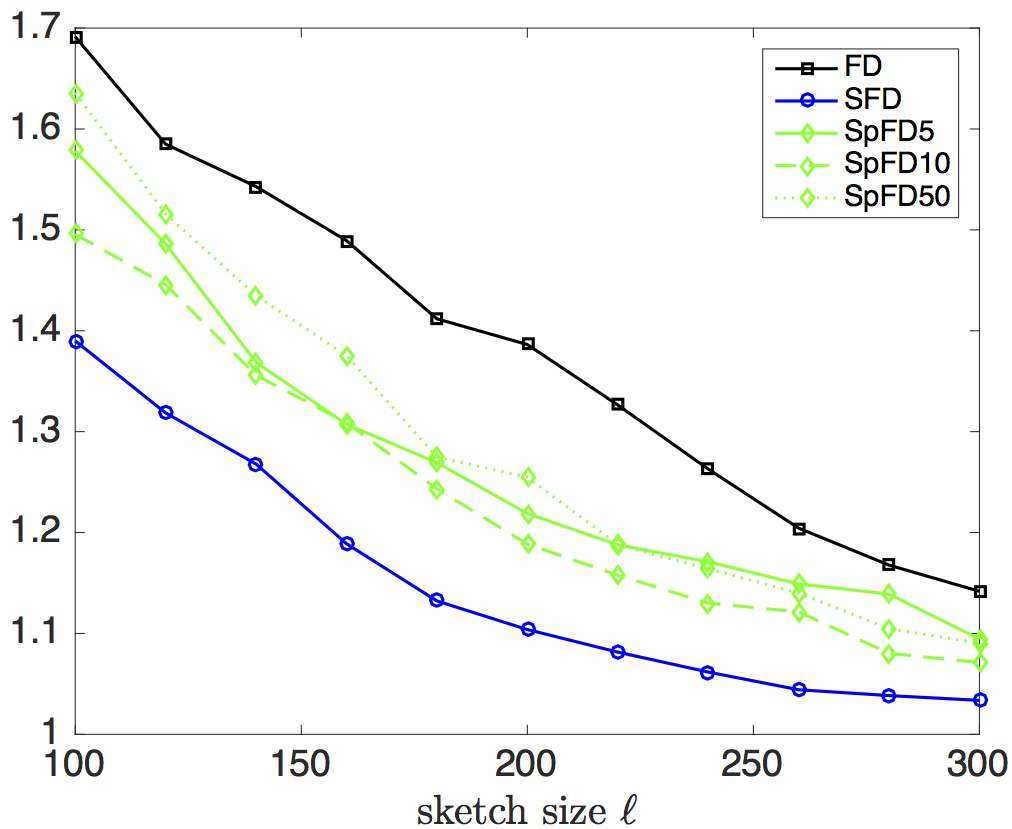}} \\
\vspace{-10pt}

\caption{Results of comparison with SFD on w8a, Birds and rcv1-small.}
\label{fig: setSFD}
\vspace{-12pt}
\end{figure*}

\end{itemize}

\ifCLASSOPTIONcompsoc
  \section*{Acknowledgments}
\else
  \section*{Acknowledgment}
\fi
The authors would like to thank the associate editor and anonymous reviewers for their
valuable comments and suggestions on earlier versions of this paper.

\ifCLASSOPTIONcaptionsoff
  \newpage
\fi

\end{document}